\documentclass[11pt,a4paper,reqno]{amsart}
\usepackage{amsmath,amsfonts,amssymb,amsthm}
\usepackage{graphicx}
\usepackage{url}
\usepackage{mathrsfs}
\usepackage[usenames,dvipsnames]{color}

\definecolor{darkred}{rgb}{0.4,0.1,0.1}
\definecolor{darkred2}{rgb}{0.6,0.1,0.1}

\usepackage[colorlinks=true,linkcolor=darkred,citecolor=Blue]{hyperref}
\numberwithin{equation}{section}
\theoremstyle{plain}
\newtheorem{thm}{Theorem}[section]
\newtheorem{lem}[thm]{Lemma}

\newtheorem{prop}[thm]{Proposition}
\newtheorem{cor}[thm]{Corollary}

\theoremstyle{definition}
\newtheorem{ex}[thm]{Example}

\theoremstyle{remark}
\newtheorem{remark}[thm]{Remark}
\theoremstyle{plain}

\newcommand{\be}{\begin{equation}}
\newcommand{\ee}{\end{equation}}
\newcommand{\beu}{\begin{equation*}}
\newcommand{\eeu}{\end{equation*}}
\newcommand{\besu}{\begin{equation*}
\begin{aligned}}
\newcommand{\eesu}{\end{aligned}
\end{equation*}}
\newcommand{\bes}{\begin{equation}
\begin{aligned}}
\newcommand{\ees}{\end{aligned}
\end{equation}}

\newcommand\cA{\mathcal A}

\newcommand\cF{\mathcal F}

\newcommand\NN{\mathbb N}

\newcommand\ov{\overline}
\newcommand\wt{\widetilde}

\newcommand{\defeq}{\mathrel{\mathop:}=}

\newcommand\rmO{{\rm O}}
\newcommand\rmo{{\rm o}}
\newcommand\rmd{\mathrm{d}}

\DeclareMathOperator\ran{ran}

\newcommand\void[1]{}

\DeclareMathOperator\Op{Op}
\DeclareMathOperator\Real{Re}

\renewcommand\Re{\Real}

      \def\dC{{\mathbb C}}

      \def\dR{{\mathbb R}}

\def\cA{{\mathcal A}}      
      \def\cF{{\mathcal F}}

\newcommand{\dom}{\mathrm{dom}\,}

\newcounter{counter_a}
\newenvironment{myenum}{\begin{list}{{\rm(\roman{counter_a})}}%
{\usecounter{counter_a}
\setlength{\itemsep}{0.5ex}\setlength{\topsep}{0.7ex}
\setlength{\leftmargin}{5ex}\setlength{\labelwidth}{5ex}}}{\end{list}}

\newcommand\ul{\underline}
\newcommand\ula{\underline{a}}

\title[Spectral asymptotics]{Spectral asymptotics for resolvent differences of
elliptic operators with {\boldmath$\delta$} and {\boldmath$\delta^\prime$}-interactions on hypersurfaces}

\author[J. Behrndt]{Jussi Behrndt}
\address{Technische Universit\"{a}t Graz\\
Institut f\"{u}r Numerische Mathematik\\
Steyrergasse 30\\
8010 Graz\\ Austria}
\email{behrndt@tugraz.at}

\author[G. Grubb]{Gerd Grubb}
\address{Department of Mathematical Sciences\\Copenhagen University \\ Universitets\-parken 5\\ DK-2100\\ Copenhagen\\ Denmark}
\email{grubb@math.ku.dk}

\author[M. Langer]{Matthias Langer}
\address{Department of Mathematics and Statistics\\
University of Strathclyde\\
26 Richmond Street\\ Glasgow G1 1XH\\ United Kingdom}
\email{m.langer@strath.ac.uk}

\author[V. Lotoreichik]{Vladimir Lotoreichik}
\address{Technische Universit\"{a}t Graz\\
Institut f\"{u}r Numerische Mathematik\\
Steyrergasse 30\\
8010 Graz\\ Austria}
\email{lotoreichik@math.tugraz.at}

\begin{document}

\begin{abstract}
We consider self-adjoint realizations of a second-order elliptic
differential expression on $\dR^n$ with singular interactions of $\delta$ and $\delta^\prime$-type
supported on a compact closed smooth hypersurface in $\dR^n$. In our main results we
prove spectral asymptotics formulae with refined remainder estimates for the singular values
of the resolvent difference between the standard self-adjoint realizations and
the
operators with a $\delta$ and $\delta^\prime$-interaction, respectively.
Our technique makes use of general pseudodifferential methods,
classical results on spectral asymptotics of $\psi$do's on closed manifolds
and Krein-type resolvent formulae.
\end{abstract}

\maketitle

\section{Introduction}

In this paper we  study self-adjoint operator realizations
of the formally symmetric, uniformly strongly elliptic differential expression
\[
\big(\cA u\big)(x) \defeq -\sum_{j,k=1}^n
\partial_j\big(a_{jk}(x)\partial_k u\big)(x) +  a(x)u(x),
 \qquad x\in\dR^n,
\]
with singular interactions of $\delta$ and $\delta^\prime$-type supported on
a $C^\infty$-smooth compact hypersurface
$\Sigma\subset\dR^n$, which splits $\dR^n$
into a bounded open set $\Omega_{-}$ and an unbounded open set $\Omega_{+}$.
More precisely, denote by $\cA_\pm$ the restrictions of $\cA$ to $\Omega_\pm$,
let $\gamma^\pm$ and $\nu^\pm$ be the
trace and conormal trace, respectively, on the boundary $\Sigma$ of $\Omega_\pm$,  let
$\alpha,\beta \in C^\infty(\Sigma)$ be real functions with $\beta(x^\prime)\not=0$
for all $x^\prime\in\Sigma$, and consider the elliptic realizations
\begin{equation*}
\begin{split}
  A_{\delta,\alpha}u & = \cA_+ u_+\oplus \cA_-u_-, \\
  \dom A_{\delta,\alpha} & =
  \bigl\{u = u_+\oplus u_-\in H^2(\Omega_+)\oplus H^2(\Omega_-)\colon \\
  &\qquad \gamma^+u_+ = \gamma^-u_-,\;\; \nu^+ u_+ + \nu^- u_- = \alpha\gamma^+ u_+\bigr\},
\end{split}
\end{equation*}
and
\begin{equation*}
\begin{split}
  A_{\delta^\prime,\beta}u & = \cA_+ u_+\oplus \cA_-u_-, \\
  \dom A_{\delta^\prime,\beta} & =\bigl\{u = u_+\oplus u_-\in H^2(\Omega_+)\oplus H^2(\Omega_-)\colon \\
  &\qquad\nu^+ u_+ + \nu^- u_- = 0,\;\; \beta\nu^+ u_+ = \gamma^+ u_+ - \gamma^- u_-\bigr\},
\end{split}
\end{equation*}
which are self-adjoint and bounded from below in $L_2(\dR^n)$; cf.\ Theorem~\ref{th:number1}. Our main goal is to compare the resolvents of
$A_{\delta,\alpha}$ and $A_{\delta^\prime,\beta}$ with the resolvent of the 'free' or 'unperturbed' self-adjoint realization
\[
A_0u = \cA u,\qquad \dom A_0 =  H^2(\dR^n),
\]
and to prove spectral asymptotics formulae with refined
remainder estimates  for the singular values of the corresponding
resolvent differences.
Without loss of generality we may assume that a sufficiently large positive
constant is added to $\cA$ such that  all operators
under consideration have a positive lower bound;
hence we consider
\begin{equation}\label{jussig}
G_{\delta,\alpha} = A_{\delta,\alpha}^{-1} - A_0^{-1}\qquad \text{and}
\qquad G_{\delta',\beta} = A_{\delta',\beta}^{-1} - A_0^{-1}.
\end{equation}
It is known that both operators $G_{\delta,\alpha}$ and $G_{\delta',\beta}$
are compact in $L_2(\dR^n)$, and estimates for the decay of the
singular values $s_k(G_{\delta,\alpha})$ and $s_k(G_{\delta',\beta})$ were recently obtained in
Behrndt et al.\ \cite{BLL13.IEOT} and  \cite{BLL13.Poincare} (for the special case $\cA=-\Delta+a$).
In our main results Theorem~\ref{th:asymp_delta} and
Theorem~\ref{thm:asymp}  we shall prove the more precise asymptotic
results with estimates of the remainder of the form
\begin{equation}\label{jussisk}
\begin{alignedat}{2}
  s_k(G_{\delta,\alpha}) & =
  C_{\delta,\alpha}k^{-\frac{3}{n-1}} + \rmO\bigl(k^{-\frac{4}{n-1}}\bigr),\qquad
  & & k\rightarrow\infty, \\
  s_k(G_{\delta',\beta}) & =
  C_{\delta'}k^{-\frac{2}{n-1}} + \rmO\bigl(k^{-\frac{3}{n-1}}\bigr),\qquad
  & & k\rightarrow\infty,
\end{alignedat}
\end{equation}
with positive constants $C_{\delta,\alpha}$ and $C_{\delta'}$ which are given
explicitly in terms of the coefficients of $\cA$ and $\alpha$; the constant $C_{\delta'}$
is independent of $\beta$.
Note that the singular values of $G_{\delta,\alpha}$ converge faster than the
singular values of $G_{\delta',\beta}$.
We mention that for the first result in \eqref{jussisk} it is assumed that the
function $\alpha$ does not vanish on $\Sigma$; if this assumption is dropped,
the estimate holds with remainder $\rmo (k^{-\frac 3{n-1} })$;
cf.\ Theorem~\ref{th:asymp_delta}.
In the course of our work we also make use of
the direct sum $A_\nu$ of the self-adjoint Neumann operators in $L_2(\Omega_+)$
and $L_2(\Omega_-)$ and we show that the singular values
of $G_{\delta',\beta,\nu} = A_{\delta',\beta}^{-1} - A_\nu^{-1}$ satisfy
\begin{equation}\label{jussisk2}
s_k(G_{\delta',\beta,\nu}) =
C_{\delta',\beta,\nu} k^{-\frac{3}{n-1}} +
\rmO\bigl(k^{-\frac{4}{n-1}}\bigr),\qquad k\rightarrow\infty,
\end{equation}
with the constant $C_{\delta',\beta,\nu} > 0$ explicitly given; cf.\ Theorem~\ref{th:asymp_delta'}.
The proofs of \eqref{jussisk} and \eqref{jussisk2} are mainly based on
pseudodifferential techniques and classical results on spectral asymptotics of $\psi$do's on closed
$C^\infty$-smooth manifolds due to Seeley \cite{S67}, H\"{o}rmander \cite{H68} and Grubb \cite{G84.Duke}.
We also refer to Boutet de Monvel \cite{BdM71}, H\"{o}rmander \cite{H71}, Taylor \cite{T81},
 Rempel and Schulze \cite{RS82} and Grubb \cite{G96, G09}
for general pseudodifferential methods.
A further ingredient in our analysis is a Krein-type resolvent formula,
which provides a factorization of the operators in \eqref{jussig} and is
discussed in detail in Section~\ref{sec:krein}; cf.\ Brasche et al.\
\cite{BEKS94}, Alpay and Behrndt \cite{AB09},
Behrndt et al.\ \cite{BLL13.IEOT,BLL13.Poincare}.

Our results in this paper contribute to a prominent field in the analysis of
partial differential operators: asymptotic estimates for the resolvent difference of
elliptic operators subject to different boundary conditions
were first obtained by Povzner \cite{P53} and Birman \cite{B62}. These estimates were sharpened to
spectral asymptotics formulae by Grubb \cite{G74} for bounded domains and by
Birman and Solomjak \cite{BS79, BS80} for exterior domains,
further generalized by Grubb \cite{G84.Duke, G84.CPDE}, and more recently in
\cite{G11.AA,G11mixed, G14};
see also the review paper \cite{G12}.
For the case of two Robin Laplacians a faster convergence of the singular values
was observed in Behrndt et al.\ \cite{BLLLP10}, and further refined to spectral asymptotics in Grubb \cite{G11.JST}.
We also list the closely related works Deift and Simon
\cite{DS75}, Bardos et al.\ \cite{BGR82}, Gorbachuk and Kutovo\u{\i} \cite{GoK82}, Brasche \cite{B01}, Carron \cite{Ca02}, Malamud \cite{M10}
and Lotoreichik and Rohleder \cite{LR12}
with spectral estimates for resolvent differences and resolvent power differences.

We wish to emphasize that the operators $A_{\delta,\alpha}$ and $A_{\delta^\prime,\beta}$ have
attracted considerable interest in the last two decades from more applied branches
of mathematics and mathematical physics.  In the special case $\cA=-\Delta+a$ we
refer to the review paper \cite{E08} by Exner for an overview on
Schr\"odinger operators with $\delta$-interactions supported on curves and hypersurfaces.
Such Hamiltonians are physically relevant in quantum mechanics, where they
are employed in many-body problems and in the description of various nanostructures,
as well as in the theory of photonic crystals; see, e.g.\
Figotin and Kuchment \cite{FK96}, Popov \cite{P96} and Brummelhuis and Duclos \cite{BD06}.
At the same time there is a  mathematical motivation to study Schr\"odinger operators with
$\delta$-interactions on hypersurfaces because
these operators exhibit non-trivial and interesting
spectral properties; for more details we refer to Brasche et al.\
\cite{BEKS94},
Exner et al.\ \cite{EHL06, EI01, EK03, EP14}, Suslina and Shterenberg
\cite{SS01},
Kondej and Veseli\'c
\cite{KV07}, Kondej and Krej\v{c}i\v{r}\'ik \cite{KK13},
Duch\^{e}ne and Raymond \cite{DR13}
and the references therein.
Schr\"odinger operators with $\delta'$-interactions supported on hypersurfaces are much
less studied than their $\delta$-counterparts.
They have been rigorously defined (in a general setting) only recently in \cite{BLL13.Poincare};
the works Behrndt et al.\ \cite{BEL13} and Exner and Jex \cite{EJ13} on their spectral properties appeared subsequently.
We also mention that for very special geometries such operators were considered
earlier in Antoine et al.\ \cite{AGS87} and Shabani \cite{S88}.

\subsection*{Acknowledgements}

J.~Behrndt and V.~Lotoreichik gratefully acknowledge financial support by the Austrian
Science Fund (FWF), project P 25162-N26. G.~Grubb and M.~Langer
are grateful for the stimulating research stay and the hospitality at the
Graz University of Technology in October 2013 where parts of this paper were written.

\section{The differential operators}
\label{ssec:diffop}

Throughout this paper let $\cA$ be the following second-order
formally symmetric differential expression on $\dR^n$:
\begin{equation}\label{defA}
  \big(\cA u\big)(x) \defeq -\sum_{j,k=1}^n
  \partial_j\big(a_{jk}(x)\partial_k u\big)(x) +  a(x)u(x),
  \qquad x\in\dR^n,
\end{equation}
with real-valued $a_{jk}\in C^\infty(\dR^n)$ satisfying $a_{jk}(x)= a_{kj}(x)$
for all $x\in\dR^n$, $j,k=1,\dots,n$, and a bounded real-valued
coefficient $a\in C^\infty(\dR^n)$.  We assume that $a_{jk}$ and all
their derivatives are bounded and that $\cA$ is uniformly strongly elliptic, i.e.\
\[
  \sum_{j,k=1}^n a_{jk}(x)\xi_j\xi_k \ge C|\xi|^2, \qquad x,\xi\in\dR^n,
\]
for some constant $C>0$.

Further, let $\Sigma\subset\dR^n$ be a $C^\infty$-smooth $(n-1)$-dimensional manifold that separates
the Euclidean space $\dR^n$ into a bounded open set $\Omega_-$ and an
unbounded open set $\Omega_+$.
In the following we denote by $u_+$ and $u_-$ the restrictions of $u\in L_2(\dR^n)$
to $\Omega_+$ and $\Omega_-$, respectively;
the restrictions of the differential expression $\cA$ to $\Omega_\pm$ are denoted by $\cA_\pm$.
For functions $u_\pm\in H^2(\Omega_\pm)$ denote by $\gamma^\pm u_\pm$ the  traces
(boundary values on $\Sigma $)
and by $\nu^\pm u_\pm$ the outward conormal derivatives of $u_\pm$:
\begin{equation}\label{nupm}
  \nu^\pm u_\pm = \sum_{j,k=1}^n a_{jk}\nu_{\pm,j}\gamma^\pm\partial_k u_\pm,
\end{equation}
where $(\nu_{\pm,1}(x),\ldots,\nu_{\pm,n}(x))$ is the exterior unit normal to $\Omega_\pm$
at $x\in\Sigma$.
If $u=u_+\oplus u_-\in H^2(\Omega_+)\oplus H^2(\Omega_-)$ and $\gamma^+u_+=\gamma^-u_-$, then
$u\in H^1(\dR^n)$ and we write $\gamma u$ for $\gamma^+u_+=\gamma^-u_-$.

Let us introduce the following operators:
the free realization of $\cA$ in $L_2(\dR^n)$,
\begin{equation}
\label{A0}
  A_0 u \defeq \cA u,\qquad
  \dom A_0 \defeq H^2(\dR^n),
\end{equation}
the Dirichlet realizations on $\Omega_+$ and $\Omega_-$,
\begin{equation*}
  A_{\pm,\gamma} u_\pm \defeq \cA_\pm u_\pm,\qquad
  \dom A_{\pm,\gamma} \defeq \bigl\{u_\pm \in H^2(\Omega_\pm)\colon \gamma^\pm u_\pm = 0\bigr\},
\end{equation*}
and the Neumann realizations,
\begin{equation*}
  A_{\pm,\nu}u_\pm \defeq \cA_\pm u_\pm,\qquad
  \dom A_{\pm,\nu} \defeq \bigl\{u_\pm \in H^2(\Omega_\pm)\colon \nu^\pm u_\pm = 0\bigr\}.
\end{equation*}
It is well known that the operators $A_0$, $A_{+,\gamma}$, $A_{-,\gamma}$, $A_{+,\nu}$
and $A_{-,\nu}$ are self-adjoint and bounded below.

Let us also introduce direct sums of operators on $\Omega_+$ and $\Omega_-$:
\begin{equation}
\label{ADNsum}
  A_\gamma \defeq A_{+,\gamma} \oplus A_{-,\gamma}, \qquad
  A_\nu \defeq A_{+,\nu} \oplus A_{-,\nu},
\end{equation}
which are self-adjoint operators in $L_2(\dR^n)=L_2(\Omega_+)\oplus L_2(\Omega_-)$.
Note also that the domain of $A_0$ can be written with interface conditions:
\begin{align*}
  \dom A_0 = \bigl\{& u = u_+\oplus u_-\in H^2(\Omega_+)\oplus H^2(\Omega_-)\colon \\
  & \gamma^+u_+ = \gamma^-u_-,\;\; \nu^+u_+ = -\nu^-u_-\bigr\}.
\end{align*}

Moreover, let us fix a real-valued function $\alpha \in C^\infty(\Sigma)$,
and define the $\delta$-operator with strength $\alpha$ by
\begin{equation}
\label{Adelta}
\begin{split}
  A_{\delta,\alpha}u &\defeq \cA_+ u_+\oplus \cA_-u_-, \\
  \dom A_{\delta,\alpha} &\defeq \bigl\{u = u_+\oplus u_-\in H^2(\Omega_+)\oplus H^2(\Omega_-)\colon \\
  &\qquad \gamma^+u_+ = \gamma^-u_-,\;\; \nu^+ u_+ + \nu^- u_- = \alpha\gamma u\bigr\}.
\end{split}
\end{equation}
Let us also fix a real-valued function $\beta\in C^\infty(\Sigma)$ such that $\beta$ is non-zero
on $\Sigma$, and define the $\delta'$-operator with strength $\beta$ by
\begin{equation}
\label{Adeltaprime}
\begin{split}
  A_{\delta^\prime,\beta}u &\defeq \cA_+ u_+\oplus \cA_-u_-, \\
  \dom A_{\delta^\prime,\beta} &\defeq \bigl\{u = u_+\oplus u_-\in H^2(\Omega_+)\oplus H^2(\Omega_-)\colon \\
  &\qquad\nu^+ u_+ + \nu^- u_- = 0,\;\; \beta\nu^+ u_+ = \gamma^+ u_+ - \gamma^- u_-\bigr\}.
\end{split}
\end{equation}

The statements in the next theorem were shown in \cite[Theorem~4.17]{BLL13.IEOT} and
\cite[Theorem~3.11, 3.14, and 3.16]{BLL13.Poincare} for the special case $\cA=-\Delta+a$;
the general case can be shown in a similar way. For the self-adjointness of $A_{\delta,\alpha}$ and $A_{\delta',\beta}$
one can also use the symmetry together
with elliptic regularity theory as done in a related situation in \cite[Theorem~7.3]{G74}.

\begin{thm}\label{th:number1}
The operators $A_{\delta,\alpha}$ and $A_{\delta',\beta}$ are self-adjoint
and bounded below in $L_2(\dR^n)$.
\end{thm}

Since all these operators are bounded below, we can assume without loss of generality
(by adding a sufficiently large real constant to $a$)
that $A_0$, $A_{\pm,\gamma}$, $A_{\pm,\nu}$, $A_{\delta,\alpha}$ and $A_{\delta',\beta}$
are positive with $0$ in the resolvent set.

We shall often tacitly identify $H^s(\Omega _+)\oplus H^s(\Omega _-)$
with $H^s(\Omega _+)\times H^s(\Omega _-)$ and write the operators in
matrix form.

\section{Pseudodifferential methods}
\label{sec:psido}

In order to show the spectral asymptotics formulae we are aiming for,
we have to go deeper into the definitions of the entering operators by
pseudodifferential techniques. Pseudodifferential operators ($\psi$do's)
$P$ are defined on $\dR^n$ by formulae
\[
  (Pu)(x) = \Op\bigl(p(x,\xi)\bigr)u
  = p(x,D)u
  = \frac{1}{(2\pi)^n}\int_{\dR^n}e^{ix\cdot \xi }p(x,\xi )\hat u(\xi )\,\rmd x,
\]
where $\hat u(\xi )=\cF u=\int e^{-ix\cdot \xi }u(x)\rmd x$ is the
Fourier transform; $p(x,\xi )$ is called the \emph{symbol} of $P$. There are
various conditions on $p$, interpretations to distributions $u$, and
rules of calculus, in particular behaviour under coordinate changes
that allow the definition on manifolds, for which we refer to the vast
literature, e.g.\ \cite{H71,T81,G09}.  The symbols we consider are
``classical'' or ``polyhomogeneous'', meaning that $p(x,\xi )$ is an asymptotic series of
functions $p_{d-j}(x,\xi )$, $j=0,1,2,\dots$, homogeneous of degree
$d-j$ in $\xi $; $p$ (and $P$) is then said to be of order $d$, and
the principal symbol is the first term $p^0=p_{d}$. It has an
invariant meaning in the manifold situation.

The following theorem is an important ingredient in the
proofs in Section~\ref{sec:4}. The first part is essentially due to
Seeley \cite{S67}.  This paper treats the elliptic case;
how the estimate can be extended to the general case is discussed,
e.g.\ in \cite[Lemma~4.5 and following paragraph]{G84.Duke}
with more references given as well.
The second part is due to H\"ormander \cite{H68}. The transition between
his formulation in terms of the counting function for $P^{-1}$ and the
eigenvalue asymptotics for $P$ is accounted for, e.g.\ in
\cite[Lemma~6.2]{G78} and \cite[Lemma~A.5]{G96}.

\begin{thm}\label{thm:grubb}
Let $P$ be a classical pseudodifferential
operator of negative order $-t$ on $\Sigma $,
with principal symbol $p^0(x,\xi)$.
Then the following statements hold.
\begin{myenum}
\item $P$ is a compact operator in $L_2(\Sigma )$, and its
singular values satisfy
\[
  s_k(P) = (c(P))^{\frac{t}{n-1}} k^{-\frac{t}{n-1}}
  + \rmo(k^{-\frac{t}{n-1}}),  \qquad k\to\infty,
\]
where
\[
  c(P) = \frac{1}{(n-1)(2\pi)^{n-1}}\int_{\Sigma}\int_{|\xi|=1}
  |p^0(x,\xi )|^{\frac{n-1}{t}}\rmd\omega(\xi)\rmd\sigma(x);
\]
here $\sigma$ and $\omega$ are the surface measures on the
hypersurfaces $\Sigma$ and $\{\xi\in\dR^{n-1}\colon |\xi|=1\}$,
respectively.
\item If, moreover, $P$ is elliptic and invertible, then the
asymptotic estimate can be sharpened to the form
\[
  s_k(P) =(c(P))^{\frac{t}{n-1}}k^{-\frac{t}{n-1}}
  + \rmO(k^{-\frac{t+1}{n-1}}),\qquad k\rightarrow\infty.
\]
\end{myenum}
\end{thm}

\medskip

A general systematic theory covering the boundary value problems we
are considering, as well as much more general
situations, was introduced by Boutet de Monvel \cite{BdM71}: the theory of
pseudodifferential boundary operators ($\psi $dbo's). Besides working
with pseudodifferential operators $P$ on $\dR^n$ and their
versions $P_+$ truncated
to smooth subsets $\Omega$ (in particular to $\dR^n_+$), the theory
includes Poisson operators $K$ (going from $\partial\Omega $ to
$\Omega $), trace operators $T$ (going from $\Omega $ to $\partial\Omega$),
$\psi$do's $S$ on  $\partial\Omega $, and the so-called singular
Green operators $G$, essentially in the form of finite sums
or infinite series: $\sum K_jT_j$.
We shall not use the $\psi$dbo calculus in full generality,
but rather its notation and elementary composition rules.
Details on the $\psi$dbo calculus are found, e.g.\ in
\cite{BdM71,RS82,G84.Duke, G96,G09}.

In this theory, the operators are described by use of local coordinate
systems, carrying the study of the operators over to the situation of
$\Omega =\dR_+^n \defeq \dR^{n-1}\times\dR_+$.
Here a differential operator $P=\sum_{|\alpha |\le m}a_\alpha (x)D^\alpha$ with the symbol
$p(x,\xi )=\sum_{|\alpha |\le m}a_\alpha (x)\xi ^\alpha$
has the principal symbol $p^0(x,\xi )=\sum_{|\alpha |=m}a_\alpha (x)\xi ^\alpha$,
and the model operator at a point $(x',0)\in \dR^{n-1}\times \{0\}$
is $p^0(x',0,\xi',D_n)=\sum_{|\alpha|=m}a_\alpha (x',0){\xi'}^{\alpha'}D_n^{\alpha_n}$.
The solution operator for the Dirichlet problem for our $\cA$ on $\Omega_+$
with non-zero boundary data, zero interior data, is a Poisson operator.
Such operators, carried over to $\dR_+^n$, are generally of the form
\[
  (K\varphi )(x)=\frac{1}{(2\pi)^{n-1}}\int_{\dR^{n-1}}e^{ix'\cdot\xi'}
  \tilde k(x',x_n,\xi')\hat \varphi (\xi')\,\rmd\xi',
  \qquad x=(x',x_n),
\]
where $\tilde k(x',x_n,\xi')$ is called the \emph{symbol-kernel} of $K$;
it is a $C^\infty $-function on $\dR_+^n\times \dR^{n-1}$ that
is rapidly decreasing for $x_n\to\infty $, with
\begin{equation}\label{deforder}
\begin{aligned}
  \sup_{x_n\in(0,\infty)}\bigl|x_n^l\partial_n^{l'}\partial_{x'}^\beta \partial_{\xi'}^\alpha
  \tilde k (x,\xi')\bigr|
  \le C_{l,l',\alpha ,\beta }(1+|\xi'|)^{d-l+l'-|\alpha|}, \hspace*{10ex}& \\
  l,l'\in\NN_0,\,\alpha,\beta\in\NN_0^{n-1},&
\end{aligned}
\end{equation}
for some $d$, and is then said to be of \emph{order} $d$.  More information
on the structure of symbol-kernels is found, e.g.\
in \cite[Section~10.1]{G09}.
The symbol-kernel $\tilde k$ is a series of terms with certain
quasi-homogeneities in $(\xi',x_n)$, corresponding to falling homogeneities in $\xi $
of the terms in the Fourier transform w.r.t.\
$x_n$, the \emph{symbol} $k=\cF_{x_n\to \xi_n}e^+\tilde k$. There is
a principal part, the top order term.

It is also known from the general calculus that the adjoint of a
Poisson operator is a trace operator (of the general form defined in
the $\psi$dbo calculus), and that a Poisson operator composed to the
left with a trace operator gives
a pseudodifferential operator on $\partial\Omega $
(on $\dR^{n-1}$ in local coordinates).

In Lemma~\ref{le.poisson} below we describe the principal symbol-kernel
of the Poisson solution operators and the principal symbols of the
Dirichlet-to-Neumann and Neumann-to-Dirichlet maps corresponding to
$\cA$ from \eqref{defA} both on $\Omega_+$ and on $\Omega_-$.
For this purpose, let us write the principal symbol
$a^0(x,\xi)$ of $\cA$ in local coordinates
$(x',x_n)=(x_1,\ldots,x_{n-1},x_n)\in\dR_+^n$
at the boundary of $\Omega_+$:
\begin{equation}\label{quadratic}
  \ula^0(x',0,\xi) = \sum_{j,k=1}^n \ula_{jk}(x')\xi_j\xi_k
  = \ula_{nn}(x')\xi_n^2 + 2b(x',\xi')\xi_n + c(x',\xi').
\end{equation}

Here $\xi'=(\xi_1,\ldots,\xi_{n-1})$, $\ul{a}_{nn}(x')>0$,
\begin{equation}\label{defbc}
  b(x',\xi') \defeq \sum_{j=1}^{n-1} \ula_{jn}(x')\xi_j, \qquad
  c(x',\xi') \defeq \sum_{j,k=1}^{n-1} \ula_{jk}(x')\xi_j\xi_k
\end{equation}
and $\ula_{nn}c>b^2$
when $\xi'\ne0$ since $\ula^0>0$ for $\xi \ne0$.
The roots of the second-order polynomial in $\xi_n$ on the right-hand side
of \eqref{quadratic} are
\[
  \lambda_{\pm}(x',\xi') = \frac{-b(x',\xi')\pm i\sqrt{\ula_{nn}(x')c(x',\xi')-(b(x',\xi'))^2}}{\ula_{nn}(x')}\,,
\]
lying in the upper, respectively lower, complex half-plane and being
homogeneous of degree $1$ in $\xi'$.  Define
\begin{equation}\label{defkappa}
\begin{aligned}
  \kappa_0 &\defeq \sqrt{\ula_{nn}c-b^2} \quad (> 0 \;\;\text{when}\;\xi'\ne0), \\[0.5ex]
  \kappa_\pm &\defeq \mp i \lambda_\pm
  = \frac{\kappa_0\pm ib}{\ula_{nn}}\,.
\end{aligned}
\end{equation}
Clearly, $\kappa_\pm$ are complex conjugates and have positive real part:
$\Re\kappa_\pm=\kappa_0/\ula_{nn}$, and satisfy $\kappa_+\ov\kappa_+=c/\ula_{nn}$.
With these expressions we can factorize~$\ula^0$:
\begin{align}
  \ula^0(x',0,\xi',\xi_n) &= \ula_{nn}(x')(\kappa_+ +i\xi_n)(\kappa_--i\xi_n),
   \\[0.5ex]
  \ula^0(x',0,\xi',D_n) &= \ula_{nn}(x')(\kappa_++\partial_n)(\kappa_--\partial_n).
  \label{a0fact2}
\end{align}

Let $K_\gamma^\pm:H^{3/2}(\Sigma)\to H^2(\Omega_\pm)$ be the
Poisson solution operators that map a $\varphi\in H^{3/2}(\Omega_\pm)$
onto the solutions $u_\pm\in H^2(\Omega_\pm)$ of the boundary value problems
\begin{equation}\label{bvp_poisson_dir}
  \cA_\pm u_\pm = 0, \quad \gamma^\pm u_\pm = \varphi.
\end{equation}
Similarly, let
$K_\nu^\pm: H^{1/2}(\Sigma)\to H^2(\Omega_\pm)$
be the Poisson solution operators corresponding to the Neumann problems
\begin{equation}\label{bvp_poisson_neu}
  \cA_\pm u_\pm = 0, \quad \nu^\pm u_\pm = \psi.
\end{equation}
Moreover, we define the Dirichlet-to-Neumann and Neumann-to-Dirichlet
operators by
\begin{equation}
\label{DNND}
  P^\pm_{\gamma,\nu} \defeq \nu^\pm K^\pm_\gamma, \qquad
  P^\pm_{\nu,\gamma} \defeq \gamma^\pm K^\pm_\nu.
\end{equation}

In the next lemma we collect properties of these operators, which are
needed in the proofs of our main results.

\begin{lem}\label{le.poisson}
Let the operators $K_\gamma^\pm$, $K_\nu^\pm$, $P_{\gamma,\nu}^\pm$ and
$P_{\nu,\gamma}^\pm$ be as above.  Then the following statements hold.
\begin{myenum}
\item
The operators $K^+_\gamma$ and $K^+_\nu$ are Poisson operators
of orders $0$ and $-1$, respectively.
Their principal symbol-kernels are, in local coordinates,
\begin{align}
  \tilde k^{+0}_\gamma (x',x_n,\xi')
  &= e^{-\kappa_+(x',\xi')x_n},
  \label{tildek+0g} \\
  \tilde k^{+0}_\nu(x',x_n,\xi')
  &= \frac{1}{\kappa_0(x',\xi')}e^{-\kappa_+(x',\xi')x_n}.
  \label{tildek+0n}
\end{align}

\item
The operators $P^+_{\gamma,\nu}$ and $P^+_{\nu,\gamma}$ are
pseudodifferential operators of orders $1$ and $-1$, respectively.
Their principal symbols are
\begin{equation}\label{p+0}
  p^{+0}_{\gamma,\nu}(x',\xi') = \kappa_0(x',\xi') \qquad\text{and}\qquad
  p^{+0}_{\nu,\gamma}(x',\xi') = \frac{1}{\kappa_0(x',\xi')}\,,\!\!
\end{equation}
which are positive.

\item
The compositions $(K_\gamma ^+)^*K_\gamma ^+$ and
$(K_\nu ^+)^*K_\nu ^+$ are pseudodifferential operators
of orders $-1$ and $-3$, respectively.
Their principal symbols are
\[
  \frac{\ula_{nn}(x')}{2\kappa_0(x',\xi')} \qquad\text{and}\qquad
  \frac{\ula_{nn}(x')}{2(\kappa_0(x',\xi'))^3}\,.
\]
\end{myenum}
For $\cA$ on $\Omega_-$ the formulae hold with $\kappa_-$ instead
of $\kappa_+$.
\end{lem}

\begin{proof}
(i)
To find the principal symbol-kernel of $K_\gamma^+$, we have to solve
the following model problem for each $(x',\xi')$ with $\xi '\ne 0$
on the one-dimensional level:
\begin{equation}\label{839}
  \ula^0(x',0,\xi',D_n)u(x_n) = 0 \quad\text{on }\dR_+, \qquad u(0)=\varphi
  \in \dC,
\end{equation}
where $\ula^0$ is the principal symbol of $\cA$ in local coordinates.
It follows from \eqref{a0fact2} and the inequality
$\Re\kappa_+>0$ that
the $L_2(\dR_+)$-solution of \eqref{839} is
\[
  u(x_n)=\varphi e^{-\kappa_+x_n}.
\]
Hence $\tilde k^{+0}_\gamma = e^{-\kappa_+x_n}$ is the principal symbol-kernel
of the Poisson operator $K_\gamma^+$. This shows \eqref{tildek+0g}.
In view of \eqref{deforder}, it has order $0$. 

Let us now consider the Neumann problem.
In local coordinates the conormal derivative takes the form
\[
  \ul\nu^+ u = -\ula_{nn}\partial_{x_n}u\big|_{x_n=0}-\sum_{k=1}^{n-1}\ula_{nk}i\xi_k u(0);
\]
cf.\ \eqref{nupm} and observe that the outward normal is $(0,\ldots,0,-1)$.
Since
\begin{equation}\label{518}
  \ul\nu^+ e^{-\kappa_+x_n} = -\ula_{nn}(-\kappa_+)-\sum_{k=1}^{n-1}\ula_{nk}i\xi_k
  = \ula_{nn}\frac{\kappa_0+ib}{\ula_{nn}}-ib = \kappa_0,
\end{equation}
the solution of $\ula^0(x',0,\xi',D_n)u(x_n)=0$, $\ul\nu^+u=\psi$ is
\[
  u(x_n) = \psi\frac{1}{\kappa_0}e^{-\kappa_+x_n},
\]
which yields \eqref{tildek+0n}.  Moreover, $\tilde k^{+0}_\nu$ has
order $-1$.

(ii)
It follows from \eqref{518} that the Dirichlet-to-Neumann operator
$P_{\gamma,\nu}^+$ has the principal symbol
$p^{+0}_{\gamma,\nu} = \ul\nu^+\tilde k_\gamma^{+0} = \kappa_0$,
which is of order $1$ since $\kappa_0$ is homogeneous
of degree $1$ in $\xi'$.
The principal symbol of $P_{\nu,\gamma}^+$
is $p^{+0}_{\nu,\gamma} = \tilde k_\nu^{+0}|_{x_n=0} = 1/\kappa_0$.

(iii)
The adjoint of $K_\gamma ^+$ has the principal part in local
coordinates acting like
\[
  u(x) \mapsto \frac{1}{(2\pi)^{n-1}}\int_{\dR^{n-1}} e^{ix'\cdot \xi'}\int_0^\infty
  e^{-\ov\kappa_+x_n}\cF_{x'\to \xi'}u(x',x_n)\, \rmd x_n\,\rmd\xi',
\]
with symbol-kernel $\overline{\tilde k_\gamma^{+0}}$ and order $-1$;
for the latter see, e.g.\ \cite[Theorem~10.29 and Remark 10.6]{G09}.
Then the composition of the
principal part of $(K^+_\gamma)^*$ with the
principal part of $K^+_\gamma $ is the $\psi$do on $\dR^{n-1}$ with symbol
\[
  \int_0^\infty e^{-\ov\kappa_+x_n}e^{-\kappa_+x_n}\, \rmd x_n
  = \frac{1}{\ov\kappa_++\kappa_+} = \frac{\ula_{nn}}{2\kappa_0} 
\]
and order $-1$ by \cite[Proposition~10.10\,(v)]{G09}.

For $K^+_\nu $, the principal part of the adjoint acts like
\[
  u(x)\mapsto \frac{1}{(2\pi)^{n-1}}\int_{\dR^{n-1}} e^{ix'\cdot \xi'}\int_0^\infty
  \frac{1}{\kappa_0} e^{-\ov\kappa_+x_n}\cF_{x'\to \xi'}u(x',x_n)\,\rmd x_n\,\rmd\xi',
\]
so we find that $(K^+_\nu)^*K^+_\nu $ has the principal symbol
\[
  \int_0^\infty \frac{1}{\kappa_0}e^{-\ov\kappa_+x_n}\frac{1}{\kappa_0}e^{-\kappa_+x_n}\rmd x_n
  = \frac{1}{\kappa_0^2}\cdot\frac{1}{\ov\kappa_++\kappa_+}
  = \frac{\ula_{nn}}{2\kappa_0^3}\,.
\]

For invariance with respect to coordinate changes
see, e.g.\ \cite[Theorem~8.1]{G09} for $\psi$do's
and \cite[Theorem~2.4.11]{G96} for $\psi$dbo's.

For the same operator $\cA$ considered on $\Omega_-$, the symbol in
local coordinates at a boundary point is as above with the direction
of $x_n$, and hence also of $\xi_n$, reverted. Then $\kappa_+$ and $\kappa_-$
exchange roles.
\end{proof}

The following lemma collects some properties of the operators $K_\gamma^\pm$ and
$K_\nu^\pm$ considered as operators between $L_2$-spaces.

\begin{lem}\label{le.adj}
Let $K_\gamma^\pm$ and $K_\nu^\pm$
be as above.
Then the $L_2$-adjoints of $K_\gamma^\pm$ and $K_\nu^\pm$ satisfy
\begin{equation}\label{adj_K}
  (K_\gamma^\pm)^* = -\nu^\pm A_{\pm,\gamma}^{-1}, \qquad
  (K_\nu^\pm)^* = \gamma^\pm A_{\pm,\nu}^{-1}.
\end{equation}
\end{lem}

\begin{proof}
Let $\varphi\in H^{3/2}(\Sigma)$ and set $u_+\defeq K_\gamma^+\varphi$,
which satisfies \eqref{bvp_poisson_dir}.
Moreover, let $f\in L_2(\Omega_+)$ and set $v_+\defeq A_{+,\gamma}^{-1}f$.
Then Green's identity implies
\begin{align*}
  (f,K_\gamma^+\varphi) &= (\cA_+v_+,u_+)
  = (\cA_+v_+,u_+)-(v_+,\cA_+u_+) \\
  &= (\gamma^+v_+,\nu^+u_+)-(\nu^+v_+,\gamma^+u_+)
  = -(\nu^+A_{+,\gamma}^{-1}f,\varphi)
\end{align*}
which yields the relation for $K_\gamma^+$ in \eqref{adj_K}
since $\nu^+A_{+,\gamma}^{-1}$ is bounded.
The other relations in \eqref{adj_K} are shown in a similar way.
\end{proof}

\begin{ex}\label{ex:laplace1}
In the special, important case
$\cA = -\Delta$ we derive from \eqref{quadratic} and \eqref{defbc} that
\[
\ula_{nn}(x') \equiv 1,\qquad b(x',\xi') \equiv 0,\qquad c(x',\xi') = |\xi'|^2;
\]
note that the principal symbol at the boundary point is unchanged under the
transformation to local coordinates.
Hence, following \eqref{defkappa} we get
\[
\kappa_\pm(x',\xi') = \kappa_0(x',\xi') = |\xi'|.
\]
Thus, according to Lemma~\ref{le.poisson} the principal symbol-kernels
of $K_\gamma^{\pm}$, $K_\nu^{\pm}$  are
\begin{equation*}
\label{Lapsymbols1}
\tilde k_\gamma^{\pm 0} = e^{-|\xi'|x_n}
\qquad\text{and}\qquad
\tilde k_\nu^{\pm 0} = \frac{e^{-|\xi'|x_n}}{|\xi'|};
\end{equation*}
the principal symbols of $P^{\pm}_{\gamma,\nu}$, $P^{\pm}_{\nu,\gamma}$
are
\begin{equation*}
\label{Lapsymbols2}
p^{\pm 0}_{\gamma,\nu} = |\xi'|
\qquad\text{and}\qquad
p^{\pm 0}_{\nu,\gamma} = \frac{1}{|\xi'|};
\end{equation*}
and the principal symbols of $(K_\gamma^\pm)^*K_\gamma^\pm$,
$(K_\nu^\pm)^*K_\nu^\pm$ are
\begin{equation*}
\label{Lapsymbols3}
\frac{1}{2|\xi'|}\qquad\text{and}\qquad \frac{1}{2|\xi'|^3}.
\end{equation*}
\end{ex}

\section{Krein-type formulae}
\label{sec:krein}

In this section we provide Krein-type formulae for differences between inverses
of self-adjoint realizations of $\cA$ defined in Section~\ref{ssec:diffop}.
First we derive Krein-type formulae in a general setting, and then we simplify
these formulae for particular $\delta$ and $\delta'$-couplings.
Similar formulae for systems acting on a single bounded domain can be
found in the paper \cite{G74}.
For coupled problems the reader may also consult \cite[Section~4]{BLL13.IEOT}
and \cite[Section~3]{BLL13.Poincare}.

Let us define the trace mapping
\begin{equation}\label{rho1}
  \varrho \colon H^2(\Omega_+)\oplus H^2(\Omega_-)
  \rightarrow H^{3/2}(\Sigma)\times H^{3/2}(\Sigma)\times H^{1/2}(\Sigma)\times H^{1/2}(\Sigma)
\end{equation}
by
\begin{equation}
\label{rho}
  \varrho(u_+\oplus u_-) \defeq \begin{pmatrix} \gamma^+u_+ \\[0.5ex] \gamma^-u_- \\[0.5ex]
  \nu^+u_+ \\[0.5ex] \nu^- u_- \end{pmatrix},
\end{equation}
where $\gamma^\pm$ are the  traces on $\Sigma $ from the two sides of
it, and $\nu^\pm$ are the outward conormal derivatives; cf.\ Section~\ref{ssec:diffop}.
Note that the map $\varrho$ is surjective by the classical trace theorem; see, e.g.\ \cite{LM68}.

Let in the following $\alpha,\beta\in C^\infty(\Sigma)$
be real-valued with $\beta$ non-vanishing on $\Sigma$.
We consider realizations of $\cA$ in $L_2(\dR^n)$ defined by
\begin{equation}
\label{A1}
\begin{split}
  A_*u &\defeq \cA_+u_+\oplus \cA_-u_-, \\
  \dom A_* &\defeq \bigl\{u\in H^2(\Omega_+)\oplus H^2(\Omega_-)\colon B_*\varrho u =0\bigr\},
\end{split}
\end{equation}
where $B_*$ is one of the following matrices
\begin{equation}
\label{B}
\begin{split}
  & B_0 = \begin{pmatrix} 1&-1&0&0 \\ 0 &0&1&1 \end{pmatrix}, \;
  B_\gamma = \begin{pmatrix} 1&0&0&0 \\ 0&1&0&0 \end{pmatrix}, \;
  B_\nu = \begin{pmatrix} 0&0&1&0 \\ 0&0&0&1 \end{pmatrix}, \\[1ex]
  & B_{\delta,\alpha} = \begin{pmatrix} 1&-1&0&0 \\ -\alpha &0&1&1 \end{pmatrix}, \;
  B_{\delta',\beta} = \begin{pmatrix} 1&-1&-\beta &0 \\ 0&0&1&1 \end{pmatrix}.
\end{split}
\end{equation}
The self-adjoint operators $A_0$, $A_\gamma$, $A_\nu$, $A_{\delta,\alpha}$
and $A_{\delta',\beta}$ defined in \eqref{A0}, \eqref{ADNsum}, \eqref{Adelta}
and \eqref{Adeltaprime} correspond, respectively, to $B_0$, $B_\gamma$,
$B_\nu$, $B_{\delta,\alpha}$ and $B_{\delta',\beta}$, in the sense of \eqref{A1}.
The boundary conditions that are induced by the matrices $B_0$, $B_\gamma$,
$B_\nu$, $B_{\delta,\alpha}$ and $B_{\delta',\beta}$
are analogues of what are called \emph{normal boundary conditions} in \cite{G74}.
A typical property of such boundary conditions is given in the lemma below, which
follows from the fact that $\varrho$ defined in \eqref{rho1}, \eqref{rho} is surjective.


\begin{lem}
\label{lem:surj}
With $\varrho$ defined by \eqref{rho1} and \eqref{rho},
let $B_*$ be one of the matrices in \eqref{B} and let us define
\begin{equation}
\label{R1}
  R_* \defeq H^{s}(\Sigma)\times H^{t}(\Sigma)
\end{equation}
with
\begin{alignat*}{2}
  & s=3/2,\; t=1/2 \quad & &\text{if}\;\; B_* = B_0,\;
  B_{\delta,\alpha},\text{ or } B_{\delta',\beta},\\
  & s=t=3/2 & &\text{if}\;\; B_* = B_\gamma, \\
  & s=t=1/2 & &\text{if}\;\; B_* = B_\nu.
\end{alignat*}
Then $B_*\varrho$ is surjective from $H^2(\Omega_+)\oplus H^2(\Omega_-)$ onto $R_*$.
\end{lem}

As discussed at the end of Section~\ref{ssec:diffop} we can assume without loss of generality
that the operators $A_0$, $A_\gamma$, $A_\nu$, $A_{\delta,\alpha}$
and $A_{\delta',\beta}$ are positive with $0$ in the resolvent set.
Then the semi-homogeneous Dirichlet and Neumann boundary value problems \eqref{bvp_poisson_dir}
and \eqref{bvp_poisson_neu} for the differential expression $\cA_\pm$ on $\Omega_\pm$
are uniquely solvable in $H^2(\Omega_\pm)$,
and $K_\gamma^\pm$ and $K_\nu^\pm$ are the corresponding solution operators.
It follows that for all
$\varphi=\binom{\varphi_+}{\varphi_-}\in H^{3/2}(\Sigma) \times H^{3/2}(\Sigma)$
the problem
\[
  \cA_+ u_+\oplus\cA_-u_- = 0, \qquad B_\gamma\varrho u
  = \varphi,
\]
with $B_\gamma$ as in \eqref{B},
has a unique solution $u=u_+\oplus u_-\in H^2(\Omega_+)\oplus H^2(\Omega_-)$.
The corresponding solution operator is
\begin{equation}
\label{Kgamma}
  K_\gamma = \begin{pmatrix} K_\gamma^+ & 0 \\ 0& K_\gamma^- \end{pmatrix},
\end{equation}
i.e.\ $u=K_\gamma\varphi$.
Note that for $u=u_+\oplus u_-\in H^2(\Omega_+)\oplus H^2(\Omega_-)$ one has
\begin{equation}\label{fixed_point}
  \cA_+u_+\oplus\cA_-u_-=0 \quad\Rightarrow\quad
  u = K_\gamma B_\gamma\varrho u.
\end{equation}
Similarly, for all $\psi=\binom{\psi_+}{\psi_-}\in H^{1/2}(\Sigma) \times H^{1/2}(\Sigma)$
the problem
\[
  \cA_+ u_+\oplus\cA_-u_- = 0, \qquad B_\nu\varrho u = \psi,
\]
with $B_\nu$ as in \eqref{B},
has a unique solution $u=u_+\oplus u_-\in H^2(\Omega_+)\oplus H^2(\Omega_-)$;
the solution operator is given by
\begin{equation}
\label{Knu}
  K_\nu = \begin{pmatrix} K_\nu^+ & 0 \\ 0& K_\nu^- \end{pmatrix}.
\end{equation}

In the next proposition we investigate the solvability of the
boundary value problems associated with the matrices $B_*$ in \eqref{B}
and derive related properties of the $2\times 2$ matrix $\psi$do's
\begin{equation}
\label{PhiPsi}
  \Phi_* \defeq B_*\varrho K_\gamma\quad\text{and}\quad\Psi_* \defeq B_*\varrho K_\nu.
\end{equation}


\begin{prop}
\label{prop:unique}
Let $B_*$ be one of the matrices in \eqref{B} with the associated
space $R_*$ as in \eqref{R1},
and let the matrix $\psi$do's $\Phi_*$ and $\Psi_*$ be defined by \eqref{PhiPsi}.
Then the following statements hold.
\begin{myenum}
\item For all $\psi\in R_*$
the boundary value problem
\begin{equation}
\label{bvp}
  \cA_+ u_+\oplus\cA_-u_- = 0, \qquad B_*\varrho u = \psi,
\end{equation}
has a unique solution  $u=u_+\oplus u_-\in H^2(\Omega_+)\oplus H^2(\Omega_-)$.
\item
The matrix $\psi$do $\Phi_*$ is bijective from
$H^{3/2}(\Sigma)\times H^{3/2}(\Sigma)$ onto $R_*$.
\item
The matrix $\psi$do $\Psi_*$ is bijective from
$H^{1/2}(\Sigma)\times H^{1/2}(\Sigma)$ onto $R_*$.
\end{myenum}
\end{prop}


\begin{proof}
In the following, $A_*$ is the self-adjoint operator in $L_2(\dR^n)$
corresponding to the matrix $B_*$ in the sense of \eqref{A1}.
By our assumptions $A_*$ is strictly positive. This implies that the
semi-homogeneous boundary value problem
\begin{equation}
\label{bvp2}
  \cA_+ w_+\oplus\cA_-w_- = f,\qquad B_*\varrho w = 0,
\end{equation}
is uniquely solvable for all $f\in L_2(\dR^n)$, and the unique solution is
given by $w \defeq A_*^{-1}f\in H^2(\Omega_+)\oplus H^2(\Omega_-)$.

(i) Let $\psi\in R_*$ and choose $v \in H^2(\Omega_+)\oplus H^2(\Omega_-)$
such that $B_*\varrho v = \psi$, which is possible by Lemma~\ref{lem:surj}.
The boundary value problem
\begin{equation*}
  \cA_+ z_+\oplus\cA_- z_- = -\cA_+  v_+\oplus\cA_-v_-,
\qquad B_*\varrho z = 0,
\end{equation*}
has a unique solution $z\in H^2(\Omega_+)\oplus H^2(\Omega_-)$; cf.\ \eqref{bvp2}.
It follows that
\[
  u \defeq z +v \in H^2(\Omega_+)\oplus H^2(\Omega_-)
\]
is a solution of \eqref{bvp}. Moreover, this solution is unique. In fact,
suppose that $\tilde u\in H^2(\Omega_+)\oplus H^2(\Omega_-)$ is also a solution
of \eqref{bvp}.  Then $u - \widetilde u \in \dom A_*$ and $A_*(u - \widetilde u) = 0$.
As $A_*$ is strictly positive we conclude that $u=\widetilde u$.


(ii)
First we verify that $\Phi_*$ is injective.
Suppose that $\Phi_*\varphi = B_*\varrho K_\gamma\varphi = 0$ holds
for some $\varphi \in H^{3/2}(\Sigma)\times H^{3/2}(\Sigma)$, $\varphi\not=0$.
Then the function $K_\gamma\varphi\not=0$ belongs to $\dom A_*$ and
satisfies $A_* K_\gamma\varphi = 0$, a contradiction to the strict positivity of $A_*$.
In order to show that $\Phi_*$ is surjective, let $\psi \in R_*$.
By item (i) the boundary value problem
\[
  \cA_+ u_+\oplus\cA_- u_- = 0,\qquad B_*\varrho u = \psi,
\]
has a unique solution $u \in H^2(\Omega_+)\oplus H^2(\Omega_-)$.
Define
\[
  \varphi \defeq B_\gamma \varrho u \in H^{3/2}(\Sigma)\times H^{3/2}(\Sigma).
\]
Note that by item (i) the boundary value problem
\[
  \cA_+ v_+\oplus\cA_-v_- = 0,\qquad B_\gamma\varrho v = \varphi,
\]
has a unique solution in $H^2(\Omega_+)\oplus H^2(\Omega_-)$.
Observe that $K_\gamma\varphi$ and $u$ both are solutions of the above problem.
Hence $u = K_\gamma\varphi$ by the uniqueness.
From
\[
  \Phi_*\varphi = B_*\varrho K_\gamma\varphi = B_*\varrho u = \psi
\]
we conclude that $\psi\in\ran\Phi_*$. It follows that $\Phi_*$ is surjective onto $R_*$.


(iii) The proof of this item is analogous to the proof of (ii).
\end{proof}


We mention that the matrix $\psi$do's $\Phi_*$ and $\Psi_*$
in \eqref{PhiPsi} are elliptic.
This is essentially a consequence of the bijectivity shown in
Proposition~\ref{prop:unique} above.

In the next theorem we give Krein-type factorizations for differences
of inverses 
between either $A_\gamma$ or $A_\nu$ and one of the operators $A_0$,
$A_\nu$, $A_{\delta,\alpha}$ and $A_{\delta',\beta}$.


\begin{thm}
\label{thm:krein1}
Let $B_*$ be one of the matrices in \eqref{B}, let $A_*$,
$A_\gamma$ and $A_\nu$ be the self-adjoint
operators in $L_2(\dR^n)$ corresponding to the matrices $B_*$, $B_\gamma$ and $B_\nu$
as in \eqref{A1}, and let $\varrho$ be the trace map in \eqref{rho1}, \eqref{rho}.
Then the following
statements are true.
\begin{myenum}
\item  The formula
\[
  A_*^{-1} - A_\gamma^{-1} = - K_\gamma\Phi_*^{-1}B_*\varrho A_\gamma^{-1}
\]
holds, where $K_\gamma$ and $\Phi_*$ are as in \eqref{Kgamma} and \eqref{PhiPsi}, respectively.
\item  The formula
\[
  A_*^{-1}-A_\nu^{-1} = -K_\nu\Psi_*^{-1}B_*\varrho A_\nu^{-1}
\]
holds, where $K_\nu$ and $\Psi_*$ are as in \eqref{Knu} and \eqref{PhiPsi}, respectively.
\end{myenum}
\end{thm}


\begin{proof}
(i)
Recall that, by our assumptions, $A_*$ and $A_\gamma$ are strictly positive.
Let $f\in L_2(\dR^n)$ and note that $u \defeq A_*^{-1}f$
is the unique solution of the semi-homogeneous boundary value problem
\[
  \cA_+ u_+\oplus \cA_-u_- = f,\qquad B_*\varrho u = 0,
\]
in $H^2(\Omega_+)\oplus H^2(\Omega_-)$.
Define the functions
\begin{equation}
\label{v}
  v \defeq A_\gamma^{-1}f \qquad\text{and}\qquad z \defeq u-v.
\end{equation}
The latter satisfies
\begin{equation}\label{bvpjussi}
  \cA_+ z_+\oplus\cA_-z_- =0,\qquad B_*\varrho z = - B_*\varrho v,
\end{equation}
which, by \eqref{fixed_point}, implies that
\begin{equation}
\label{z}
  z = K_\gamma B_\gamma\varrho z.
\end{equation}
Now it follows from \eqref{z}, \eqref{PhiPsi}, \eqref{bvpjussi} and \eqref{v}
that
\begin{align*}
  z &= K_\gamma B_\gamma\varrho z
  = K_\gamma\Phi_*^{-1}B_*\varrho K_\gamma B_\gamma\varrho z
  = K_\gamma\Phi_*^{-1}B_*\varrho z \\
  &= -K_\gamma\Phi_*^{-1}B_*\varrho v
  = -K_\gamma\Phi_*^{-1}B_*\varrho A_\gamma^{-1}f.
\end{align*}
This, together with $A_*^{-1}f=u=v+z$ yields the formula in (i).

Item (ii) can be proved in the same way as item (i) when $A_\gamma$, $B_\gamma$, $K_\gamma$
and $\Phi_*$ are replaced by $A_\nu$, $B_\nu$, $K_\nu$ and $\Psi_*$, respectively.
\end{proof}


In the next proposition we simplify the formula from Theorem~\ref{thm:krein1}\,(i)
for the difference 
between the inverses of the self-adjoint operators $A_{\delta,\alpha}$ and $A_\gamma$.
In the formulation and the proof of this proposition we employ the column operator
\begin{equation}
\label{wtK}
  \wt K_\gamma \defeq \begin{pmatrix} K^+_\gamma \\[0.5ex] K^-_\gamma\end{pmatrix}:
  H^{3/2}(\Sigma)\to H^2(\Omega_+)\times H^2(\Omega_-).
\end{equation}

\begin{prop}
\label{prop:krein1}
Let $A_{\delta,\alpha}$ and $A_\gamma$ be as above, and
let $P^\pm_{\gamma,\nu}$ be the Dirichlet-to-Neumann maps in \eqref{DNND}
with principal symbols $p_{\gamma,\nu}^{\pm0}$ in \eqref{p+0}.
Then the following statements hold.

\begin{myenum}
\item
The $\psi$do $P^+_{\gamma,\nu} + P^-_{\gamma,\nu} -\alpha$ is elliptic of order $1$
with principal symbol $p_{\gamma,\nu}^{+0} + p_{\gamma,\nu}^{-0}=2\kappa_0$,
and it maps $H^{3/2}(\Sigma)$ bijectively onto $H^{1/2}(\Sigma)$.
\item
The formula
\[
  A_{\delta,\alpha}^{-1} - A_\gamma^{-1} =
  \wt K_\gamma\big(P^+_{\gamma,\nu} + P^-_{\gamma,\nu} -\alpha\big)^{-1}\wt K_\gamma^*
\]
holds, where $\wt K_\gamma$ is as in \eqref{wtK} and $\wt K_\gamma^*$ is the $L_2$-adjoint
of $\wt K_\gamma$.
\end{myenum}

\end{prop}


\begin{proof}

(i) Both operators $P^+_{\gamma,\nu}$ and
$P^-_{\gamma,\nu}$ are symmetric first-order elliptic $\psi$do's
on $\Sigma$ (see Lemmas~\ref{le.poisson}\,(ii) and \ref{le.adj}) with
principal symbols $p_{\gamma,\nu}^{\pm0}$ as in \eqref{p+0}.
Hence the $\psi$do
\begin{equation}\label{psidojussi}
 P^+_{\gamma,\nu} + P^-_{\gamma,\nu} - \alpha
\end{equation}
is also symmetric of order $1$ with principal symbol
$p_{\gamma,\nu}^{+0} + p_{\gamma,\nu}^{-0} = 2\kappa_0$.
According to Lemma~\ref{le.poisson}\,(ii) we have
$p_{\gamma,\nu}^{\pm0} > 0$ for $\xi '\ne 0$,
and hence \eqref{psidojussi} is elliptic.
By \cite[Theorem~8.11]{G09} the index of the $\psi$do \eqref{psidojussi} as a
mapping from $H^{3/2}(\Sigma)$ to $H^{1/2}(\Sigma)$ is $0$.
Hence, in order to prove bijectivity of \eqref{psidojussi}
from $H^{3/2}(\Sigma)$ onto $H^{1/2}(\Sigma)$, it remains to show
that $\ker(P^+_{\gamma,\nu} + P^-_{\gamma,\nu} -\alpha)$ is trivial.
Suppose for a moment that this were
not the case.  Then, by \cite[Theorem~8.11]{G09},
there exists a non-trivial $\varphi \in C^\infty(\Sigma)$ such that
\[
  (P^+_{\gamma,\nu} + P^-_{\gamma,\nu} -\alpha)\varphi =0.
\]
Let us consider the non-trivial function $u \defeq \wt K_\gamma \varphi$.
Then we have $u \in H^2(\Omega_+)\oplus H^2(\Omega_-)$,
$\cA_+u_+\oplus \cA_- u_- = 0$, and
\begin{align*}
  \gamma^+u_+ -\gamma^-u_- &=0, \\
  \nu^+ u_+ + \nu^- u_- -\alpha\gamma u
  &=  \nu^+ K_\gamma^+ \varphi + \nu^-K_\gamma^-\varphi -\alpha\varphi \\
  & = P^+_{\gamma,\nu} \varphi + P^-_{\gamma,\nu}\varphi -\alpha\varphi = 0.
\end{align*}
From \eqref{Adelta} we conclude that $u \in \dom A_{\delta,\alpha}$
and, moreover, $A_{\delta,\alpha}u = 0$, which contradicts the
invertibility
of $A_{\delta,\alpha}$.


(ii) Let us consider the matrix $\psi$do $\Phi_{\delta,\alpha}=B_{\delta,\alpha}\varrho K_\gamma$
(cf.\ \eqref{PhiPsi}).
For $\binom{\varphi_+}{\varphi_-}\in H^{3/2}(\Sigma)\times H^{3/2}(\Sigma)$ we have
\begin{align*}
  \Phi_{\delta,\alpha}
  \begin{pmatrix}\varphi_+ \\ \varphi_-\end{pmatrix}
  &
  = \begin{pmatrix} 1 & -1 & 0 & 0 \\ -\alpha &0 &1&1 \end{pmatrix}
  \varrho K_\gamma
  \begin{pmatrix}\varphi_+ \\ \varphi_-\end{pmatrix} \\
  &= \begin{pmatrix} 1 & -1 & 0 & 0 \\ -\alpha &0 &1&1 \end{pmatrix}
  \begin{pmatrix}
    \varphi_+ \\ \varphi_- \\[0.5ex]
    P^+_{\gamma,\nu}\varphi_+ \\[0.5ex] P^-_{\gamma,\nu}\varphi_-
    \end{pmatrix}
  \\
  &= \begin{pmatrix}
    \varphi_+ - \varphi_- \\[1ex]
    -\alpha\varphi_+ + P^+_{\gamma,\nu}\varphi_+ + P^-_{\gamma,\nu}\varphi_-
  \end{pmatrix},
\end{align*}
and hence $\Phi_{\delta,\alpha}$ can be written in matrix form as
\begin{equation*}
  \Phi_{\delta,\alpha}
  = \begin{pmatrix} 1 & -1 \\
  -\alpha+ P^+_{\gamma,\nu} & P^-_{\gamma,\nu}\end{pmatrix}.
\end{equation*}
By item (i) the operator $P^+_{\gamma,\nu} + P^-_{\gamma,\nu} -\alpha$
is bijective from $H^{3/2}(\Sigma)$ onto $H^{1/2}(\Sigma)$.
It follows that the matrix operator
\begin{equation}
\label{inv1}
  (P^+_{\gamma,\nu} + P^-_{\gamma,\nu}-\alpha)^{-1}
  \begin{pmatrix}P^-_{\gamma,\nu} & 1 \\[1ex] \alpha - P^+_{\gamma,\nu} & 1\end{pmatrix}
\end{equation}
is well defined as a mapping from $H^{3/2}(\Sigma)\times H^{1/2}(\Sigma)$
into $H^{3/2}(\Sigma)\times H^{3/2}(\Sigma)$
and that it is the inverse of $\Phi_{\delta,\alpha}$. Indeed, we have
\begin{align*}
  & (P^+_{\gamma,\nu} + P^-_{\gamma,\nu}-\alpha)^{-1}
  \begin{pmatrix}P^-_{\gamma,\nu} & 1 \\ \alpha - P^+_{\gamma,\nu} & 1\end{pmatrix}
  \begin{pmatrix} 1 & -1\\ -\alpha+ P^+_{\gamma,\nu} & P^-_{\gamma,\nu}\end{pmatrix} \\[1ex]
  &= (P^+_{\gamma,\nu} + P^-_{\gamma,\nu}-\alpha)^{-1}
  \begin{pmatrix} P^+_{\gamma,\nu} + P^-_{\gamma,\nu} -\alpha & 0 \\
  0 & P^+_{\gamma,\nu} + P^-_{\gamma,\nu} -\alpha\end{pmatrix}
  = \begin{pmatrix} I & 0 \\ 0 & I \end{pmatrix}
\end{align*}
on $H^{3/2}(\Sigma)\times H^{3/2}(\Sigma)$.

For $f=f_+\oplus f_-\in L_2(\dR^n)$ we obtain from Theorem~\ref{thm:krein1},
\eqref{adj_K}, \eqref{inv1} and \eqref{B} that
\begin{align*}
  A_{\delta,\alpha}^{-1}f - A_\gamma^{-1}f
  &= -\begin{pmatrix} K_\gamma^+ & 0 \\
  0 & K_\gamma^- \end{pmatrix}
  \Phi_{\delta,\alpha}^{-1}B_{\delta,\alpha}\varrho A_\gamma^{-1}f
  \\
  &= -\begin{pmatrix} K_\gamma^+ & 0 \\
  0 & K_\gamma^- \end{pmatrix}
  \Phi_{\delta,\alpha}^{-1}B_{\delta,\alpha}
  \begin{pmatrix} 0 \\ 0 \\ \nu^+A_{+,\gamma}^{-1}f_+ \\[0.7ex]
  \nu^-A_{-,\gamma}^{-1}f_- \end{pmatrix}
  \displaybreak[0]\\
  &= \begin{pmatrix} K_\gamma^+ & 0 \\
  0 & K_\gamma^- \end{pmatrix}
  \Phi_{\delta,\alpha}^{-1}
  \begin{pmatrix} 1 & -1 & 0 & 0 \\ -\alpha & 0 & 1 & 1 \end{pmatrix}
  \begin{pmatrix} 0 \\ 0 \\ (K_\gamma^+)^*f_+ \\[0.7ex]
  (K_\gamma^-)^*f_- \end{pmatrix}
  \displaybreak[0]\\
  &= \begin{pmatrix} K_\gamma^+ & 0 \\
  0 & K_\gamma^- \end{pmatrix}
  \Phi_{\delta,\alpha}^{-1}
  \begin{pmatrix} 0 \\ \wt K_\gamma^*f \end{pmatrix}
  \\
  &= \begin{pmatrix}
  K_\gamma^+(P_{\gamma,\nu}^++P_{\gamma,\nu}^--\alpha)^{-1}\wt K_\gamma^*f \\[0.7ex]
  K_\gamma^-(P_{\gamma,\nu}^++P_{\gamma,\nu}^--\alpha)^{-1}\wt K_\gamma^*f
  \end{pmatrix}
  \\[0.5ex]
  &= \wt K_\gamma(P_{\gamma,\nu}^++P_{\gamma,\nu}^--\alpha)^{-1}\wt K_\gamma^*f,
\end{align*}
which proves item (ii).
\end{proof}

Note that the operator $A_{\delta,\alpha}$ with  $\alpha \equiv 0$ coincides
with the operator $A_0$; cf.\ Section~\ref{ssec:diffop}.
This observation, together with Proposition~\ref{prop:krein1} and the relation
\begin{equation}\label{diff_inv_p}
\begin{aligned}
  &\bigl(P^+_{\gamma,\nu}
  + P^-_{\gamma,\nu}-\alpha\bigr)^{-1}-\bigl(P^+_{\gamma,\nu}
  + P^-_{\gamma,\nu}\bigr)^{-1} \\
  &\hspace*{10ex}= \bigl(P^+_{\gamma,\nu} + P^-_{\gamma,\nu}-\alpha\bigr)^{-1}\alpha
  \bigl(P^+_{\gamma,\nu} + P^-_{\gamma,\nu}\bigr)^{-1},
\end{aligned}
\end{equation}
yields the following corollary.

\begin{cor}\label{cor:krein.delta.0}
Under the assumptions of Proposition~{\rm \ref{prop:krein1}} and with $A_0$
as in \eqref{A0},
\[
  A_{\delta,\alpha}^{-1} - A_0^{-1}
  = \wt K_\gamma\bigl(P^+_{\gamma,\nu} + P^-_{\gamma,\nu} -\alpha\bigr)^{-1}
  \alpha\bigl(P^+_{\gamma,\nu} + P^-_{\gamma,\nu}\bigr)^{-1}\wt K_\gamma^*
\]
holds.
\end{cor}


In the next proposition we simplify the formula from Theorem~\ref{thm:krein1}\,(ii)
for the difference between the inverses of the self-adjoint operators $A_{\delta',\beta}$ and $A_\nu$.
The proof follows the same strategy as the proof of Proposition~\ref{prop:krein1}.
For the convenience of the reader we provide the essential arguments.
We also mention that
the formula in item (ii) below is similar to the one in \cite[Theorem 3.11\,(ii)]{BLL13.Poincare}.
Here we employ the column operator
\begin{equation}
\label{wtKnu}
  \wt K_\nu \defeq \begin{pmatrix} K^+_\nu \\[0.5ex] -K^-_\nu\end{pmatrix}:
  H^{1/2}(\Sigma)\to H^2(\Omega_+)\times H^2(\Omega_-).
\end{equation}

\begin{prop}
\label{prop:krein2}
Let $A_\nu$ and $A_{\delta',\beta}$ be as defined above,
and let $P^\pm_{\nu,\gamma}$ be the Neumann-to-Dirichlet maps in \eqref{DNND}.
Then the following statements hold.
\begin{myenum}
\item
The $\psi$do $\beta - (P^+_{\nu,\gamma} + P^-_{\nu,\gamma})$ is elliptic of order $0$
with principal symbol $\beta$, and it maps $H^{1/2}(\Sigma)$
bijectively onto $H^{1/2}(\Sigma)$.
\item
The formula
\[
  A_{\delta',\beta}^{-1} - A_\nu^{-1} = \wt K_\nu
\big(\beta - (P^+_{\nu,\gamma} + P^-_{\nu,\gamma})\big)^{-1}\wt K_\nu^*
\]
holds, where  $\wt K_\nu$ is as in \eqref{wtKnu} and $\wt K_\nu^*$ is
the $L_2$-adjoint of $\wt K_\nu$.
\end{myenum}
\end{prop}
\begin{proof}

(i)
The operators $P^\pm_{\nu,\gamma}$ are symmetric elliptic $\psi$do's
on $\Sigma$ of order $-1$ (see Lemmas~\ref{le.poisson}\,(ii) and \ref{le.adj}).
Since $\beta$ is real-valued and non-zero on $\Sigma$,
$\beta - (P^+_{\nu,\gamma} + P^-_{\nu,\gamma})$ is a
symmetric and elliptic $\psi$do of order $0$ with principal symbol $\beta$.
Hence, by \cite[Theorem~8.11]{G09}, its index as
a mapping from $H^{1/2}(\Sigma)$ into $H^{1/2}(\Sigma)$ is $0$.
Therefore it suffices to verify the injectivity  of
$\beta - (P^+_{\nu,\gamma} + P^-_{\nu,\gamma})$.
Suppose that this were not the case.
As in the proof of Proposition~\ref{prop:krein1}\,(i) it follows that there exists a non-trivial
$\psi \in C^\infty(\Sigma)$ such that
$(\beta - (P^+_{\nu,\gamma} + P^-_{\nu,\gamma}))\psi = 0$
and the function $\wt K_\nu \psi\not= 0$ belongs to $\ker
A_{\delta',\beta}$; this is a contradiction to the invertibility
of $A_{\delta',\beta}$.

(ii) A simple calculation shows that the matrix $\psi$do $\Psi_{\delta',\beta}=B_{\delta',\beta}\varrho K_\nu$ acts as
\begin{equation*}
  \Psi_{\delta',\beta}\begin{pmatrix} \psi_+ \\ \psi_- \end{pmatrix}
  = \begin{pmatrix} 1&-1&-\beta&0 \\ 0 &0&1&1 \end{pmatrix}\varrho K_\nu
  \begin{pmatrix} \psi_+ \\ \psi_- \end{pmatrix} =  \begin{pmatrix}
    P_{\nu,\gamma}^+ -\beta & -P_{\nu,\gamma}^- \\
    1 & 1
  \end{pmatrix}\begin{pmatrix} \psi_+ \\ \psi_- \end{pmatrix},
\end{equation*}
and a similar consideration as in the proof of Proposition~\ref{prop:krein1}~(ii) yields
\begin{equation}
\label{inv2}
  \Psi_{\delta',\beta}^{-1}=\big(P_{\nu,\gamma}^+ + P_{\nu,\gamma}^- -\beta\big)^{-1}
  \begin{pmatrix}
  1 & P_{\nu,\gamma}^- \\[1ex]
  -1& P_{\nu,\gamma}^+ -\beta
  \end{pmatrix}.
\end{equation}
It is seen from
the form of $B_{\delta',\beta}$ in \eqref{B}, relations \eqref{adj_K},
\eqref{inv2} and Theorem~\ref{thm:krein1}\,(ii) that
\begin{align*}
  & A_{\delta',\beta}^{-1}f - A_\nu^{-1}f
  = -\begin{pmatrix} K_\nu^+ & 0 \\
  0 & K_\nu^-\end{pmatrix}
  \Psi_{\delta',\beta}^{-1}B_{\delta',\beta}
  \begin{pmatrix} (K_\nu^+)^* f_+ \\ (K_\nu^-)^* f_- \\ 0 \\ 0 \end{pmatrix}\\
  &= -
  \begin{pmatrix} K_\nu^+ & 0 \\
  0 & K_\nu^- \end{pmatrix}
  \big(P^+_{\nu,\gamma} + P^-_{\nu,\gamma} -\beta\big)^{-1}
  \begin{pmatrix} (K_\nu^+)^*f_+ - (K_\nu^-)^*f_- \\[1ex]
  (K_\nu^-)^*f_- -(K_\nu^+)^*f_+ \end{pmatrix}\\
  &=\wt K_\nu \big(\beta - (P^+_{\nu,\gamma} + P^-_{\nu,\gamma})\big)^{-1}
  \wt K_\nu^*f
\end{align*}
holds for all $f\in L_2(\dR^n)$, which proves (ii).
\end{proof}


Finally we provide a more explicit formula for the difference
between the inverses of
the self-adjoint operators $A_0$ and $A_\nu$. Again the proof follows the same strategy as
the proofs of Propositions \ref{prop:krein1} and \ref{prop:krein2}.


\begin{prop}
\label{prop:krein3}
Let $A_0$ and $A_\nu$ be as above, and let $P^\pm_{\nu,\gamma}$ be
the Neumann-to-Dirichlet maps $P^\pm_{\nu,\gamma}$ in \eqref{DNND}
with principal symbols $p_{\nu,\gamma}^{\pm0}$ in \eqref{p+0}.
Then the following statements hold.
\begin{myenum}
\item
The $\psi$do $P^+_{\nu,\gamma} + P^-_{\nu,\gamma}$ is elliptic of
order $-1$ with principal symbol $p_{\nu,\gamma}^{+0} + p_{\nu,\gamma}^{-0}$,
and it maps $H^{1/2}(\Sigma)$ bijectively onto $H^{3/2}(\Sigma)$.
\item
The formula
\[
  A_\nu^{-1} - A_0^{-1}  = \wt K_\nu\big(P^+_{\nu,\gamma} + P^-_{\nu,\gamma}\big)^{-1}\wt K_\nu^*
\]
holds, where $\wt K_\nu$ is as in \eqref{wtKnu}.
\end{myenum}
\end{prop}
\begin{proof}
(i) Following the arguments in the proofs of Proposition~\ref{prop:krein1} and Proposition~\ref{prop:krein2} we
conclude that the $\psi$do $P^+_{\nu,\gamma} + P^-_{\nu,\gamma}$ is elliptic of order $-1$
with principal symbol
$p_{\nu,\gamma}^{+0}+p_{\nu,\gamma}^{-0}$, and its index as a mapping
from $H^{1/2}(\Sigma)$ into $H^{3/2}(\Sigma)$ is $0$. Again it is sufficient for the bijectivity to verify
that $\ker(P^+_{\nu,\gamma} + P^-_{\nu,\gamma})$ is trivial.
Suppose that this were not the case.
Then it follows that  there exists a non-trivial
$\psi \in C^\infty(\Sigma)$ such that
$(P^+_{\nu,\gamma} + P^-_{\nu,\gamma})\psi =0$ and the
function $\wt K_\nu \psi\not= 0$ belongs to $\ker A_0$, a
contradiction to the invertibility
of $A_0$.


(ii) The matrix $\psi$do $\Psi_0=B_0 \varrho K_\nu$ and its inverse have the form
\[
\Psi_0 = \begin{pmatrix} P_{\nu,\gamma}^+& -P_{\nu,\gamma}^-\\ 1&1\end{pmatrix}\quad\text{and}\quad
\Psi_0^{-1}=\big(P_{\nu,\gamma}^+ + P_{\nu,\gamma}^-\big)^{-1}
  \begin{pmatrix}
  1 & P_{\nu,\gamma}^- \\[1ex]
  -1 & P_{\nu,\gamma}^+
  \end{pmatrix}.
\]
Hence it follows from \eqref{adj_K}, the form of $B_0$ in \eqref{B} and
Theorem~\ref{thm:krein1}~(ii) that
\begin{align*}
  A_\nu^{-1}f - A_0^{-1}f
  &= \begin{pmatrix} K_\nu^+ & 0 \\
  0 & K_\nu^-\end{pmatrix}
  \Psi_0^{-1}B_0
  \begin{pmatrix} (K_\nu^+)^* f_+ \\ (K_\nu^-)^* f_- \\ 0 \\ 0 \end{pmatrix} \\
  &= \begin{pmatrix} K_\nu^+ & 0 \\
  0 & K_\nu^- \end{pmatrix}
  \big(P^+_{\nu,\gamma} + P^-_{\nu,\gamma}\big)^{-1}
  \begin{pmatrix} (K_\nu^+)^*f_+ - (K_\nu^-)^*f_- \\[1ex]
  (K_\nu^-)^*f_- -(K_\nu^+)^*f_+  \end{pmatrix}\\
  &=\wt K_\nu \big(P^+_{\nu, \gamma}
  + P^-_{\nu,\gamma}\big)^{-1}\wt K_\nu^*f
\end{align*}
holds for all $f\in L_2(\dR^n)$, which proves (ii).
\end{proof}

\begin{cor}\label{cor:krein.delta'.0}
Under the assumptions of Proposition~{\rm \ref{prop:krein3}}
and with $A_{\delta',\beta}$ as in \eqref{Adeltaprime},
\[
  A_{\delta',\beta}^{-1} - A_0^{-1}
  = \wt K_\nu\big(P_{\nu,\gamma}^+ + P_{\nu,\gamma}^-\big)^{-1}\beta
  \big(\beta - (P_{\nu,\gamma}^+ + P_{\nu,\gamma}^-) \big)^{-1}\wt K_\nu^*
\]
holds.
\end{cor}

\section{Spectral asymptotics for resolvent differences}\label{sec:4}

In this section we present and prove the main results of this note, namely,
we obtain spectral asymptotics formulae for the differences
between the inverses of the
operators
$A_0$, $A_\nu$, $A_{\delta,\alpha}$ and $A_{\delta',\beta}$
introduced in Section~\ref{ssec:diffop}.  These
asymptotics refine some spectral estimates for resolvent differences
found in \cite{BLL13.IEOT,BLL13.Poincare}.
The proofs are based on the Krein-type resolvent formulae proved in
Section~\ref{sec:krein}, spectral asymptotics for $\psi$do's
on smooth manifolds without boundary
and some elements of the $\psi$dbo calculus.

In all theorems of this section we suppose that the assumptions at
the beginning of Section~\ref{ssec:diffop} hold.
Moreover, let $\ula_{nn}$ and $\kappa_0$ be defined as in \eqref{quadratic}
and \eqref{defkappa}.
The next theorem contains one of our main results:
the spectral asymptotics of the difference
between the inverses of $A_{\delta,\alpha}$ and $A_0$.

\begin{thm}\label{th:asymp_delta}
Let $\alpha\in C^\infty(\Sigma)$ be real-valued,
let $A_{\delta,\alpha}$ be the self-adjoint operator in \eqref{Adelta}
and let $A_0$ be the free operator in \eqref{A0}.
Then
\[
 A_{\delta,\alpha}^{-1} - A_0^{-1}
\]
is a compact operator in $L_2(\dR^n)$
and the following two statements hold.
\begin{itemize}\setlength{\parskip}{1.2ex}
\item [\rm (i)]
The singular values $s_k$ of $A_{\delta,\alpha}^{-1} - A_0^{-1}$ satisfy
\begin{equation*}
  s_k =
  C_{\delta,\alpha}k^{-\frac{3}{n-1}} +
  \rmo\bigl(k^{-\frac{3}{n-1}}\bigr),\qquad  k\rightarrow \infty,
\end{equation*}
with the constant $ C_{\delta,\alpha} = \bigl(C_{\delta,\alpha}'\bigr)^{\frac{3}{n-1}}$,
\begin{equation*}
  C_{\delta,\alpha}' =
 \frac{1}{(n-1)(2\pi)^{n-1}}\int_{\Sigma}\int_{|\xi'|=1}\!
  \left(\frac{\ula_{nn}(x')|\alpha(x')|}{4\bigl(\kappa_0(x',\xi')\bigr)^3}
  \right)^{\!\!\frac{n-1}{3}} \!\!\rmd\omega(\xi')\rmd\sigma(x').
\end{equation*}

\item [\rm (ii)]
If\, $\alpha(x') \ne 0$ for all $x'\in\Sigma$, then the singular values $s_k$ of $A_{\delta,\alpha}^{-1} - A_0^{-1}$ satisfy
\begin{equation*}
  s_k =
  C_{\delta,\alpha} k^{-\frac{3}{n-1}}+
  \rmO\bigl(k^{-\frac{4}{n-1}}\bigr), \qquad k\rightarrow\infty.
\end{equation*}
\end{itemize}
\end{thm}
\begin{proof}
(i) Let us set
\[
  G_{\delta,\alpha} \defeq A_{\delta,\alpha}^{-1} -
  A_0^{-1}
  \quad
  \text{and}\quad
  S_\alpha\defeq\bigl(P_{\gamma,\nu}^+ + P_{\gamma,\nu}^- -
  \alpha\bigr)^{-1}\alpha
  \bigl(P_{\gamma,\nu}^++P_{\gamma,\nu}^-\bigr)^{-1},
\]
where $P_{\gamma,\nu}^+$ and $P_{\gamma,\nu}^-$ are defined in \eqref{DNND},
and let $\wt K_\gamma$ be as in \eqref{wtK}.
It follows from Corollary~\ref{cor:krein.delta.0} that
\begin{equation}\label{jussi1}
  G_{\delta,\alpha}
  = \widetilde K_\gamma S_\alpha \widetilde K_\gamma^*,
\end{equation}
which is a bounded self-adjoint operator in $L_2(\dR^n)$.  We also make use of the operator
\begin{equation*}
  R_\gamma\defeq\widetilde K_\gamma^*\widetilde K_\gamma
  = (K_\gamma^+)^*K_\gamma^+ + (K_\gamma^-)^*K_\gamma^-,
\end{equation*}
which is a $\psi$do on $\Sigma$ of order $-1$
with the principal symbol $\ula_{nn}/\kappa_0$ according to
Lemma~\ref{le.poisson}\,(iii).
Note that for a non-trivial $f\in L_2(\Sigma)$
one has
\[
((K^\pm_\gamma)^*K^\pm_\gamma f,f)_{L_2(\Sigma)}=
\|K^\pm_\gamma f\|^2_{L_2(\Omega _\pm)} > 0.
\]
Thus the non-negative self-adjoint operators $(K^\pm_\gamma)^*K^\pm_\gamma$ are
invertible; hence so is $R_\gamma $.

In the following we use that in the pseudodifferential boundary
operator calculus, $\widetilde K_\gamma $, $\widetilde K_\gamma ^*$
and $R_\gamma $ extend to continuous operators
\begin{alignat*}{2}
  \widetilde K_\gamma&\colon H^s(\Sigma)\to H^{s+\frac12}(\Omega_+)
  \times H^{s+\frac12}(\Omega_-) \qquad & &\text{for } s\in\mathbb{R}, \\
  \widetilde K_\gamma^*&\colon H^{t}(\Omega_+)
  \times H^{t} (\Omega_-) \to H^{t+\frac12}(\Sigma) \qquad & &\text{for } t>-\tfrac12, \\
  R_\gamma&\colon H^s(\Sigma)\to H^{s+1}(\Sigma) \qquad & &\text{for } s\in\mathbb{R}
\end{alignat*}
(we use the same notation for the extended/restricted operators).
By Proposition~\ref{prop:krein1}~(i) and
Lemma~\ref{le.poisson}\,(ii) the $\psi$do's
$(P_{\gamma,\nu}^++P_{\gamma,\nu}^-)^{-1}$ and
$(P_{\gamma,\nu}^++P_{\gamma,\nu}^--\alpha)^{-1}$ are both of order $-1$ and
have the principal symbol $1/(2\kappa_0)$.
Hence $S_\alpha$ is of order $-2$ and has the principal symbol $\alpha/(4\kappa_0^2)$.
It extends to a mapping
\begin{equation*}
  S_\alpha \colon H^s(\Sigma )\to H^{s+2}(\Sigma) \qquad \text{for } s\in\mathbb{R}.
\end{equation*}
It now follows from \eqref{jussi1} that
\begin{equation*}
\begin{split}
 G_{\delta,\alpha}^* G_{\delta,\alpha} &=
 \bigl(\widetilde K_\gamma S_\alpha \widetilde K_\gamma^*\bigr)^*\widetilde K_\gamma S_\alpha \widetilde K_\gamma^*\\
 &=\widetilde K_\gamma S_\alpha^* \widetilde K_\gamma^*\widetilde K_\gamma S_\alpha \widetilde K_\gamma^*\\
 &=\widetilde K_\gamma  S_\alpha  R_\gamma  S_\alpha \widetilde K_\gamma^*,
\end{split}
\end{equation*}
where we used the symmetry of $S_\alpha$, which follows from \eqref{diff_inv_p}.
Therefore (here $s_k(T)$ denotes the $k$th singular value and $\lambda_k(T)$
the $k$th positive eigenvalue of an operator $T$)
\begin{equation}\label{skcomp}
\begin{split}
  \bigl(s_k(G_{\delta,\alpha})\bigr)^2
  &= \lambda_k(G^*_{\delta,\alpha}G_{\delta,\alpha}) \\
  &=\lambda_k\bigl(\widetilde K_\gamma  S_\alpha  R_\gamma  S_\alpha
  \widetilde K_\gamma^*\bigr) \\
  &=\lambda_k\bigl( S_\alpha  R_\gamma  S_\alpha
  \widetilde K_\gamma^*\widetilde K_\gamma\bigr) \\
  &= \lambda_k(S_\alpha R_\gamma S_\alpha R_\gamma)\\
  &= \lambda_k\bigl(R_\gamma^{1/2}  S_\alpha  R_\gamma^{1/2} R_\gamma^{1/2}
  S_\alpha  R_\gamma^{1/2}\bigr)\\
  &= \bigl(s_k\bigl( R_\gamma^{1/2} S_\alpha  R_\gamma^{1/2}\bigr)\bigr)^2,
\end{split}
\end{equation}
that is, the singular values $s_k$ of $A_{\delta,\alpha}^{-1} - A_0^{-1}$ are
\begin{equation*}
  s_k(G_{\delta,\alpha})=s_k\bigl(R_\gamma^{1/2}  S_\alpha R_\gamma^{1/2}\bigr).
\end{equation*}

Here since $R_\gamma $ is  non-negative and invertible, the operator
$R_\gamma^{1/2}$ is well defined as an elliptic $\psi$do of order $-1/2$ with principal symbol
$(\ula_{nn}/\kappa_0)^{1/2}$, by \cite{S67}. Moreover,
$ S_\alpha$ is a $\psi$do of order $-2$ with principal symbol
$\alpha/(4\kappa_0^2)$; so
standard rules of the $\psi$do calculus yield that the $\psi$do
\begin{equation}
\label{Pdeltaalpha}
  P_{\delta,\alpha}\defeq  R_\gamma^{1/2}  S_\alpha  R_\gamma^{1/2}
\end{equation}
is of order $-3$ and with the principal symbol $(\ula_{nn}\alpha)/(4\kappa_0^3)$.
Now Theorem~\ref{thm:grubb}\,(i) implies that
\begin{equation*}
  s_k(P_{\delta,\alpha})
  =   C_{\delta,\alpha} k^{-\frac{3}{n-1}}
	  + \rmo\bigl(k^{-\frac{3}{n-1}}\bigr),\qquad k\rightarrow\infty,
\end{equation*}
with $C_{\delta,\alpha}$ as in the formulation of the theorem.

(ii)
Recall that $\ula_{nn}(x') > 0$ and $\kappa_0(x',\xi') > 0$
for $\xi' \ne 0$; cf.\ Section~\ref{sec:psido}.
Hence the assumption $\alpha(x^\prime)\not=0$ for all $x^\prime\in\Sigma$ implies that
the $\psi$do $P_{\delta,\alpha}$ in \eqref{Pdeltaalpha} is elliptic.
Moreover, since $S_\alpha $ is then invertible and
$R_\gamma^{1/2}$ is invertible,
the operator $P_{\delta,\alpha}$ in \eqref{Pdeltaalpha} is invertible.
Now Theorem~\ref{thm:grubb}\,(ii) implies the assertion.
\end{proof}

The next theorem contains our second main result:
the spectral asymptotics of the difference
between the inverses of $A_{\delta^\prime,\beta}$ and $A_0$.
It turns out that the principal term in the asymptotics is independent
of $\beta$. The proof of Theorem~\ref{thm:asymp} is very
similar to the  proof of Theorem~\ref{th:asymp_delta}.

\begin{thm}
\label{thm:asymp}
Let $\beta\in C^\infty(\Sigma)$ be real-valued such that
$\beta(x')\ne 0$ for all $x'\in\Sigma$,
let $A_{\delta^\prime,\beta}$ be the self-adjoint operator in \eqref{Adeltaprime} and let $A_0$ be the free operator in \eqref{A0}. Then
\[
  A_{\delta',\beta}^{-1} - A_0^{-1}
\]
is a compact operator in $L_2(\dR^n)$, and its singular values satisfy
\begin{equation*}
s_k=
 C_{\delta'} k^{-\frac{2}{n-1}}
  + \rmO\bigl(k^{-\frac{3}{n-1}}\bigr),
  \qquad k\rightarrow\infty,
\end{equation*}
with the constant $C_{\delta'} = (C_{\delta'}')^{\frac{2}{n-1}}$,
\begin{equation*}
  C_{\delta'}' = \frac{1}{(n-1)(2\pi)^{n-1}}
  \int_{\Sigma}\int_{|\xi'|=1}\left(\frac{\ula_{nn}(x')}{2(\kappa_0(x^\prime,\xi^\prime))^2}
  \right)^{\frac{n-1}{2}}\rmd\omega(\xi')\rmd\sigma(x').
\end{equation*}
\end{thm}
\begin{proof}
According to Corollary~\ref{cor:krein.delta'.0} we have
\begin{equation*}
  A_{\delta',\beta}^{-1} - A_0^{-1}=  \wt K_\nu\big(P_{\nu,\gamma}^+ + P_{\nu,\gamma}^-\big)^{-1}\beta
  \big(\beta - (P_{\nu,\gamma}^+ + P_{\nu,\gamma}^-) \big)^{-1}\wt K_\nu^*,
 \end{equation*}
where $P^+_{\nu,\gamma}$, $P^-_{\nu,\gamma}$ and  $\widetilde K_\nu$ are
as in \eqref{DNND} and \eqref{wtKnu}, respectively.
Let us set
\[
  R_\nu \defeq \widetilde K_\nu^*\widetilde K_\nu\quad\text{and}\quad
  T_\beta \defeq
  \big(P_{\nu,\gamma}^+ + P_{\nu,\gamma}^-\big)^{-1}\beta
  \big(\beta - (P_{\nu,\gamma}^+ + P_{\nu,\gamma}^-) \big)^{-1}.
\]
Since $K_\nu ^\pm$ are injective, the non-negative self-adjoint
operators $(K_\nu^\pm)^*K_\nu^\pm$ are invertible;
hence so is $R_\nu$.
As in the proof of Theorem~\ref{th:asymp_delta} we make use of the fact
that $\widetilde K_\nu$, $\widetilde K_\nu^*$ and $T_\beta$
extend to continuous operators between the respective spaces.
Then the same computation as in \eqref{skcomp} implies
\begin{equation*}
  s_k(A_{\delta',\beta}^{-1} - A_0^{-1})=s_k\bigl(R_\nu^{1/2} T_\beta R_\nu^{1/2}\bigr).
\end{equation*}
The operator $R_\nu^{1/2}$ is well defined as an elliptic $\psi$do of order $-3/2$ with principal symbol
$(\ula_{nn}/\kappa_0^3)^{1/2}$ by \cite{S67}.
It follows from Lemma~\ref{le.poisson}\,(ii) and Proposition~\ref{prop:krein2}\,(i)
that $T_\beta$ is a $\psi$do of order $1$ with principal symbol $\kappa_0/2$.
By standard rules of the calculus the $\psi$do
\begin{equation}
\label{Pdeltabeta}
  P_{\delta',\beta} \defeq R_\nu^{1/2} T_\beta R_\nu^{1/2}
\end{equation}
is of order $-2$ and has the principal symbol $\ula_{nn}/(2\kappa_0^2)$.
Since $\ula_{nn}(x') > 0$ and $\kappa_0(x',\xi') > 0$
for $\xi' \ne 0$, it follows that
the operator  $P_{\delta',\beta}$ is elliptic.
Since $T_\beta$ and $R_\nu^{1/2}$ are invertible the same true for $P_{\delta',\beta}$ in \eqref{Pdeltabeta}.
Now Theorem~\ref{thm:grubb}\,(ii) yields the spectral asymptotics
\begin{equation*}
  s_k(P_{\delta',\beta})
  =  C_{\delta'}k^{-\frac{2}{n-1}}
  + \rmO\bigl(k^{-\frac{3}{n-1}}\bigr),\qquad k\rightarrow\infty,
\end{equation*}
with $C_{\delta^\prime}$ as in the formulation of the theorem.
\end{proof}

Finally we state a third result on spectral asymptotics.
In the next theorem  the difference of the inverses of $A_{\delta',\beta}$ and the orthogonal sum of the Neumann
operators $A_\nu$ is considered. The proof is similar to the proofs of the previous theorems.
Therefore we just give the main arguments.

\begin{thm}\label{th:asymp_delta'}
Let $\beta\in C^\infty(\Sigma)$ be real-valued such that
$\beta(x')\ne0$ for all $x'\in\Sigma$, let $A_{\delta^\prime,\beta}$ be the
self-adjoint operator in \eqref{Adeltaprime} and let $A_\nu$ be as in
\eqref{ADNsum}. Then
\[
  A_{\delta',\beta}^{-1} - A_\nu^{-1}
\]
is a compact operator in $L_2(\dR^n)$, and its singular values satisfy
\begin{equation*}
s_k = C_{\delta',\beta,\nu}k^{-\frac{3}{n-1}}+\rmO\bigl(k^{-\frac{4}{n-1}}\bigr),
  \qquad k\rightarrow\infty,
\end{equation*}
with the constant
 $C_{\delta',\beta,\nu} = (C_{\delta',\beta,\nu}')^{\frac{3}{n-1}}$,
\begin{equation*}
 C_{\delta',\beta,\nu}' = \frac{1}{(n-1)(2\pi)^{n-1}}
  \int_{\Sigma}\int_{|\xi'|=1}\!\!\left(
  \frac{\ula_{nn}(x')}{|\beta(x')|\,(\kappa_0(x',\xi'))^3}\right)^{\!\frac{n-1}{3}}
  \!\!\!\rmd\omega(\xi')\rmd\sigma(x').
\end{equation*}
\end{thm}

\begin{proof}
According to Proposition~\ref{prop:krein2}\,(ii) we have
\begin{equation*}
  A_{\delta',\beta}^{-1} - A_\nu^{-1}=\widetilde
  K_\nu\bigl(\beta - (P^+_{\nu,\gamma} + P^-_{\nu,\gamma})\bigr)^{-1}
  \widetilde K_\nu^*,
\end{equation*}
where $P^+_{\nu,\gamma}$, $P^-_{\nu,\gamma}$ and $\widetilde K_\nu$
are as in \eqref{DNND} and \eqref{wtKnu}, respectively.
Let
\[
  R_\nu\defeq\widetilde K_\nu^*\widetilde K_\nu\quad\text{and}\quad
  S_\beta\defeq\bigl(\beta - (P^+_{\nu,\gamma}+P^-_{\nu,\gamma})\bigr)^{-1},
\]
and observe that $\widetilde K_\nu$, $\widetilde K_\nu^*$ and $S_\beta$
extend to continuous operators between the respective spaces.
As in the proofs of Theorems~\ref{th:asymp_delta} and \ref{thm:asymp} one verifies
\begin{equation*}
 s_k( A_{\delta',\beta}^{-1} - A_\nu^{-1})=s_k\bigl(R_\nu^{1/2} S_\beta R_\nu^{1/2}\bigr).
\end{equation*}
Furthermore, $R_\nu^{1/2}$ is a $\psi$do of order $-3/2$
with principal symbol $(\ula_{nn}/\kappa_0^3)^{1/2}$, and by Proposition~\ref{prop:krein2}\,(i)
the $\psi$do $S_\beta$ is
of order $0$ with principal symbol $\beta^{-1}$.
Therefore the $\psi$do
\begin{equation*}
    P_{\delta',\beta,\nu} \defeq R_\nu^{1/2} S_\beta R_\nu^{1/2}
\end{equation*}
is of order $-3$ and has the principal symbol $\ula_{nn}/(\beta\kappa_0^3)$.
It also follows that $P_{\delta',\beta,\nu}$ is elliptic and invertible,
and hence Theorem~\ref{thm:grubb}\,(ii) yields
the spectral asymptotics
\[
  s_k(P_{\delta',\beta,\nu}) =
C_{\delta',\beta,\nu} k^{-\frac{3}{n-1}}+
  \rmO\bigl(k^{-\frac{4}{n-1}}\bigr),\qquad k\rightarrow\infty,
\]
with $C_{\delta',\beta,\nu}$ of the form as stated in the theorem.
\end{proof}

The constants $C_{\delta,\alpha}'$, $C_{\delta'}'$ and $C_{\delta',\beta,\nu}'$
in the previous theorems can be computed explicitly
in the special case $\cA=-\Delta$.

\begin{ex}\label{ex:laplace2}
Denote by $|\Sigma|$  the area of $\Sigma$
and by $|S_{n-1}|$ the area of the $(n-1)$-dimensional
unit sphere (for $n=2$ we have $|S_{n-1}|=2$).
According to Example~\ref{ex:laplace1}
in the special case $\cA = -\Delta$ the constant
$C_{\delta,\alpha}'$
in Theorem~\ref{th:asymp_delta} is
\[
\begin{split}
  C_{\delta,\alpha}' &
  = \frac{1}{(n-1)(2\pi)^{n-1}}
  \int_{\Sigma}\int_{|\xi'|=1}
  \Bigg(\frac{1\cdot|\alpha(x')|}{4|\xi'|^3}\Bigg)^{\frac{n-1}{3}}
  \rmd\omega(\xi')\rmd\sigma(x') \\
  &= \frac{|S_{n-1}|}{(n-1)(2\pi)^{n-1}4^{(n-1)/3}}
  \int_{\Sigma}|\alpha(x')|^{\frac{n-1}{3}}\rmd\sigma(x'),
\end{split}
\]
the constant $C_{\delta'}'$ in Theorem~\ref{thm:asymp} is
\[
\begin{split}
 C_{\delta'}' & =
  \frac{1}{(n-1)(2\pi)^{n-1}}
  \int_{\Sigma}\int_{|\xi'|=1}
  \Bigg(\frac{1}{2|\xi'|^2}\Bigg)^{\frac{n-1}{2}}
  \rmd\omega(\xi')\rmd\sigma(x') \\
  &=  \frac{|\Sigma|\cdot|S_{n-1}|}{(n-1)(2\pi)^{n-1}2^{(n-1)/2}}\,,
\end{split}
\]
and the constant $C_{\delta',\beta,\nu}'$
in Theorem~\ref{th:asymp_delta'} is
\[
\begin{split}
  C_{\delta',\beta,\nu}'
  & = \frac{1}{(n-1)(2\pi)^{n-1}}
  \int_{\Sigma}\int_{|\xi'|=1}
  \Bigg(\frac{1}{|\beta(x')|\cdot|\xi'|^3}\Bigg)^{\frac{n-1}{3}}
  \rmd\omega(\xi')\rmd\sigma(x')\\
  &=
  \frac{|S_{n-1}|}{(n-1)(2\pi)^{n-1}}
  \int_{\Sigma}\Bigg(\frac{1}{|\beta(x')|}\Bigg)^{\frac{n-1}{3}}
  \rmd\sigma(x').
\end{split}
\]
\end{ex}

\begin{remark}\label{remdelta}
We note that the Krein-type formulae in Section~\ref{sec:krein} for the
differences of the inverses $A_{\delta,\alpha}^{-1}-A_0^{-1}$,
$A_{\delta^\prime,\beta}^{-1}-A_0^{-1}$ and $A_{\delta^\prime,\beta}^{-1}-A_\nu^{-1}$
can be generalized for the differences of the resolvents
$(A_{\delta,\alpha}-\lambda)^{-1}-(A_0-\lambda)^{-1}$,
$(A_{\delta^\prime,\beta}-\lambda)^{-1}-(A_0-\lambda)^{-1}$
and $(A_{\delta^\prime,\beta}-\lambda)^{-1}-(A_\nu-\lambda)^{-1}$ for all $\lambda$
in the respective resolvent sets.
Making use of such resolvent formulae one can show that the spectral asymptotics
in Theorem~\ref{th:asymp_delta}, Theorem \ref{thm:asymp} and
Theorem~\ref{th:asymp_delta'} remain true for all $\lambda$ in the respective resolvent sets.
\end{remark}


\end{document}